\documentclass[11pt]{article}

\usepackage{amsmath, amsfonts, amsthm, amssymb, mathtools}
\usepackage{authblk, color, bm, bbm, graphicx, epstopdf}
\usepackage[labelfont = bf, labelsep = period]{caption} %Option: font = small
\usepackage{subcaption, xspace, enumitem}
\usepackage{booktabs, tabularx, multirow, rotating, array}
\usepackage[bottom,hang]{footmisc}
\newcolumntype{C}{>{\centering\arraybackslash}X}
\newcolumntype{R}{>{\raggedleft\arraybackslash}X}
\newcolumntype{L}{>{\raggedright\arraybackslash}X}

\newtheorem{prop}{Proposition}

\usepackage{fullpage}[2cm]
\newcommand{\abs}[1]{\left\lvert#1\right\rvert}
\newcommand{\ubar}[1]{\underline{#1}}
\newcommand{\obar}[1]{\bar{#1}}
\DeclarePairedDelimiter\ceil{\lceil}{\rceil}     
\DeclarePairedDelimiter\floor{\lfloor}{\rfloor}

\usepackage{tikz}
\usepackage{pgfplots,tikz-3dplot}
\usetikzlibrary{calc,arrows.meta,positioning}

\usepackage{float,algorithm}
\usepackage{algpseudocode}
\algrenewcommand\algorithmicindent{2em}%
\algrenewcommand\algorithmicrequire{\textbf{Input:}}
\algnewcommand{\IIf}[1]{\State\algorithmicif\ #1\ \algorithmicthen}
\algnewcommand{\EndIIf}{\unskip\ \algorithmicend\ \algorithmicif}

\usepackage{url}

\usepackage[page]{appendix}

\newcommand{\hvrpSol}{\left(\mathbf{R}, \bm{\kappa}\right)}
\newcommand{\hvrpSolx}[1]{\left(\mathbf{R}^{#1}, \bm{\kappa}^{#1}\right)}

\title{Robust optimization of a broad class of heterogeneous vehicle routing problems under demand uncertainty}

\author[1]{Anirudh Subramanyam}
\author[2,3]{Panagiotis P.~Repoussis}
\author[1]{Chrysanthos E.~Gounaris}
\affil[1]{\small Carnegie Mellon University, Pittsburgh, United States}
\affil[2]{\small Stevens Institute of Technology, Hoboken, United States}
\affil[3]{\small Athens University of Economics and Business, Athens, Greece}
\date{}

\begin{document}

\maketitle

\begin{abstract}
This paper studies robust variants of an extended model of the classical Heterogeneous Vehicle Routing Problem (HVRP), where a mixed fleet of vehicles with different capacities, availabilities, fixed costs and routing costs is used to serve customers with uncertain demand. This model includes, as special cases, all variants of the HVRP studied in the literature with fixed and unlimited fleet sizes, accessibility restrictions at customer locations, as well as multiple depots. Contrary to its deterministic counterpart, the goal of the robust HVRP is to determine a minimum-cost set of routes and fleet composition that remains feasible for all demand realizations from a pre-specified uncertainty set. To solve this problem, we develop robust versions of classical node- and edge-exchange neighborhoods that are commonly used in local search and establish that efficient evaluation of the local moves can be achieved for five popular classes of uncertainty sets. The proposed local search is then incorporated in a modular fashion within two metaheuristic algorithms to determine robust HVRP solutions. The quality of the metaheuristic solutions is quantified using an integer programming model that provides lower bounds on the optimal solution. An extensive computational study on literature benchmarks shows that the proposed methods allow us to obtain high quality robust solutions for different uncertainty sets and with minor additional effort compared to deterministic solutions.

\end{abstract}

\section{Introduction}\label{sec:introduction}
Vehicle routing problems involve the determination of cost-optimal transportation plans for the distribution of goods or the delivery of services between production facilities, warehouses, distribution centers and end customers.
The goal is to determine an optimal assignment of customer orders to vehicles, as well as the optimal sequencing of customer orders served by individual vehicles.
The most common objective is to minimize the transportation cost, which is often expressed as the sum of one time costs (e.g., rental or capital amortization costs of individual vehicles) that are proportional to the size of the vehicle fleet, and recurring costs (e.g., fuel, labor or insurance costs) that are proportional to the total distance or duration traveled by individual vehicles.
We refer the reader to~\cite{GoldenRaghavanWasil,TothVigo:2ndEdition} for a general overview of the practical applications and numerous variants of vehicle routing problems, which differ in terms of the time scales involved, as well as in terms of the objectives and operational constraints considered.

The most widely studied variant is the Capacitated Vehicle Routing Problem (CVRP) \cite{Laporte2009:fifty_years_of_VRP}.
The CVRP aims to determine the optimal delivery of goods from a depot to a set of customers using capacity-constrained vehicles.
This archetypal variant makes two simplifying assumptions: \textit{(i)} the vehicle fleet is assumed to be homogeneous, fixed and stationed at a single depot, and \textit{(ii)} the problem data, such as customer demands and transportation costs, are assumed to be precisely known when the problem is to be solved.
However, these assumptions are often difficult to justify in practice.

First, the vehicle fleet is rarely homogeneous in most practical applications.
We refer the reader to~\cite{Hoff2010:industrial_aspects_of_hvrp} who provide an excellent overview of practical aspects of fleet sizing and dimensioning that arise in real industrial applications.
The authors argue that a vehicle fleet that is acquired over a long period of time is often heterogeneous not only because the acquired vehicles inherently have different physical characteristics, but also because they develop different characteristics over their lifetimes (e.g., operating, maintenance and insurance costs vary depending on the level of depreciation and usage).
Moreover, distributors typically want a diverse vehicle fleet, both due to operational constraints (e.g., physical dimensions or compatibility constraints that restrict access of certain vehicles to certain areas) as well as the inherent benefits of owning a versatile fleet.

Second, the assumption of deterministic problem data is unrealistic.
Indeed, the parameters of a vehicle routing problem are often subject to significant uncertainty, and their precise values are often only observed gradually during the execution of the transportation plan.
For example, travel and service times can vary due to unforeseen events such as bad weather, mechanical breakdowns or traffic congestion.
Similarly, customer demands fluctuate from day to day and in fact, they may be uncertain even at the time when the vehicles are to be dispatched.
The motivation for taking into account this uncertainty is particularly strong when making strategic or tactical fleet composition decisions.
This is because, on the one hand, operational parameters such as customer demands and travel times are often not known with certainty at the strategic level when the fleet composition is to be decided.
On the other hand, these long-term decisions often involve significant amounts of capital investment; therefore, when customer demand is higher than expected, external vehicle fleets must be leased over a short-term operational horizon, which is also costly.

In this paper, we depart from the aforementioned assumptions of fixed, homogeneous fleets and deterministic parameters.
In particular, we study a generalization of the CVRP, known as the Heterogeneous Vehicle Routing Problem (HVRP), and focus our attention to problem settings where the customer demands are subject to uncertainty.
The HVRP can be used to model a number of problem settings including those with an unlimited number of vehicles, accessibility restrictions at customer locations, vehicle-dependent fixed costs and routing costs, as well as multiple depots.
The HVRP was introduced in the seminal work of~\cite{Golden1984:fsm}, and since then, several papers have studied this problem.
A recent, comprehensive literature review of the HVRP and its numerous variants, solution algorithms and applications can be found in~\cite{Koc2016:thirty_years_of_HVRP,Irnich2014}.
Among the 150 or so works reviewed in these papers, only a single work~\cite{Teodorovic1995} has attempted to address the effect of uncertainty.

To date, most contributions to vehicle routing under uncertainty, including the aforementioned work of~\cite{Teodorovic1995}, model the uncertain customer demands as random variables that follow a known probability distribution.
The models take the form of stochastic programs~\cite{BirgeLouveaux:SP} or Markov decision processes~\cite{Bertsekas:DP}, and the goal is to optimize a risk measure (such as the expected value or the conditional-value-at-risk) of the transportation costs, subject to satisfying the constraints with high probability.
We refer to~\cite{Gendreau2014:svrp} for an excellent survey of the stochastic vehicle routing literature.
According to the classification in~\cite{Gendreau2014:svrp}, the three most common modeling paradigms are recourse models, chance-constrained models and reoptimization models.
In recourse models, a planned or \textit{here-and-now} solution must be designed before the true values of the uncertain parameters become known while recourse or \textit{wait-and-see} actions can be taken in the future after observing their true values.
Different to recourse models, the goal of chance-constrained models is to determine a single set of vehicle routes such that the probability of the total demand served along each route exceeding the vehicle capacity is within prespecified limits.
Finally, in reoptimization models, the goal is to dynamically modify the vehicle routes during their execution, as uncertain information gets revealed over time.

Despite their success in addressing a wide variety of decision problems under uncertainty, the aforementioned models suffer from two shortcomings. First, they assume that the true probability distribution is known precisely, which is rarely the case in practice. Second, they are plagued by the curse of dimensionality, which often makes them computationally intractable for inputs of realistic size. Indeed, simply evaluating the expected costs involves multi-dimensional integration, which is difficult to perform in the context of an optimization search process.

Robust optimization is a promising methodology that attempts to address these shortcomings.
In contrast to stochastic programs and Markov decision processes, robust optimization does not assume precise knowledge of the probability distribution governing the uncertain parameters.
In fact, it only requires knowledge of their support and assumes a deterministic set-based model of the uncertainty.
The basic robust optimization model consists of determining a solution that remains feasible for any realization of the uncertain parameters in a given set, which is referred to as the \textit{uncertainty set}.
The primary advantage of robust optimization is its computational tractability, since it is well known that it can be reformulated as a nonstochastic model that enjoys tractability properties similar to the deterministic problem.
This makes it well suited to address large problem instances where little or no historical information about the uncertain parameters is available.
We refer the reader to~\cite{RobustOptimizationBook,BertsimasBrownCaramanis2011:SiamReview} for an overview of the theory and applications of robust optimization.

Over the last decade, several vehicle routing problems have been addressed using robust optimization.
Most studies have focused on the CVRP under demand and/or travel time uncertainty~\cite{Sungur2008,Sungur2010,Ordonez2010,Erera2010,Gounaris2013:OR,Gounaris2016:AMP}, but a few have also addressed the CVRP with Time Windows under travel time uncertainty~\cite{Agra2013:robust_vrptw,Zhang2016} and the Multi-Period CVRP under customer order uncertainty~\cite{Subramanyam2017:robust_mpvrp}.

\subsection{Our Contributions}\label{sec:introduction:contributions}
In this paper, we study the modeling and solution of the robust HVRP under customer demand uncertainty.
The goal is to determine a single set of vehicle routes as well as the associated fleet size and composition such that the total demand served on any route is less than the associated vehicle capacity, under any realization of the demands in a prespecified uncertainty set.
Our paper generalizes the works of~\cite{Gounaris2013:OR,Gounaris2016:AMP} for the robust CVRP along multiple directions.
First, our work addresses not only the CVRP, but also all major variants of the HVRP that have been considered in the literature.
Second, we consider three new families of practically-relevant uncertainty sets in addition to the two considered in~\cite{Gounaris2013:OR,Gounaris2016:AMP}, and we discuss how each of these five sets can be constructed in the context of vehicle routing using historical data.

The distinct contributions of this paper may be summarized as follows.
\begin{enumerate}
\item [-] We augment various node- and edge-exchange neighborhoods that are commonly used in local search algorithms so that they generate routes that remain capacity-feasible for any anticipated demand realization. Although the evaluation of each local move amounts, in general, to the solution of a convex optimization problem, we show how this can be done more efficiently via closed-form expressions for five popular classes of uncertainty sets, namely budget sets, factor models, ellipsoids, cardinality-constrained sets, and discrete sets. To this purpose, we present suitable data structures and establish time and storage complexities for updating those during local search.

\item [-] We demonstrate that our robust local search method can be easily incorporated into any local search based metaheuristic that is devised for deterministic vehicle routing. To that end, we focus on two largely different metaheuristic algorithms, namely Iterated Local Search and Adaptive Memory Programming, and point them to utilize our robust local search so as to determine robust feasible routes. Furthermore, we show that the computational tractability of these two algorithms remains on par with that of their deterministic equivalents.

\item [-] We propose an integer programming formulation and associated branch-and-cut algorithm to obtain lower bounds on the optimal robust HVRP solution. A key feature of this formulation is a generalization of the rounded capacity inequalities from the CVRP to the robust HVRP, and we show how the efficient separation of these generalized inequalities is enabled by the same closed-form expressions and data structures that are used in the robust local search.

\item [-] We conduct an extensive computational study using the five aforementioned uncertainty sets and literature benchmarks of several problem variants, including the fleet size and mix, fixed fleet, multi-depot and site-dependent vehicle routing problems. We elucidate the computational overhead of incorporating robustness in metaheuristic algorithms, the quality of the lower bounds from the exact algorithm, as well as the trade-off between cost and robustness against the uncertainty sets considered.
\end{enumerate}

The rest of this paper is organized as follows.
Section~\ref{sec:problem_definition} provides a mathematical definition of the various HVRP variants that we consider in this paper;
Section~\ref{sec:wc_evaluation} presents the examined uncertainty sets, the closed-form expressions of the worst-case load as well as data structures for their efficient computation;
Section~\ref{sec:local_search} presents the robust local search and its incorporation into metaheuristics;
Section~\ref{sec:milp_and_bnc} presents the integer programming formulation and branch-and-cut algorithm;
Section~\ref{sec:results} presents computational results; and, Section~\ref{sec:conclusions} offers concluding remarks.

\section{Robust Heterogeneous Vehicle Routing}\label{sec:problem_definition}
An undirected graph $G = (V, E)$ with nodes $V = \{0, 1, \ldots, n\}$ and edges $E$ is given.
The node $0 \in V$ represents the depot, whereas each node $i\in V_C \coloneqq V\setminus\{0\}$ represents a customer with demand $q_i \in \mathbb{R}_{+}$.
The depot is equipped with a heterogeneous fleet of vehicles, which is composed of a set $K = \{1, \ldots, m\}$ of $m$ different \emph{vehicle types}.
For each type $k \in K$, $m_k$ vehicles are available, each of which has capacity $Q_k$.
Furthermore, each vehicle of type $k \in K$ incurs a \emph{fixed cost} $f_k \in \mathbb{R}_{+}$ if it is used and a \emph{routing cost} $c_{ijk} \in \mathbb{R}_{+}$ if it traverses the edge $(i, j) \in E$.

A \emph{route} is a simple cycle in $G$ that passes through the depot.
We represent a route by $R = \left(r_1, \ldots, r_{\abs{R}}\right)$, where $r_l \in V_C$ represents the $l^\text{th}$ customer and $\abs{R}$ the number of customers visited on route $R$.
We use the notation $i \in R$ to indicate that customer $i$ is visited on route $R$; that is, $i = r_l$ for some $l \in \left\{1, \ldots, \abs{R} \right\}$.
If route $R$ is performed by a vehicle of type $k \in K$, then a cost equal to the sum of the routing costs and fixed cost associated with that vehicle type is incurred, namely $c(R,k) = f_k + \sum_{l=0}^{\abs{R}} c_{r_lr_{l+1}k}$, where we have defined $r_0 = r_{\abs{R}+1} = 0$.

When the customer demands are known precisely, a \emph{set of routes} $\mathbf{R} = (R_1, \ldots, R_H)$ in conjunction with a \emph{fleet composition vector} $\bm{\kappa} = (\kappa_1, \ldots, \kappa_H)$ is said to define a \emph{feasible} HVRP solution $\hvrpSol$ if and only if the following conditions are satisfied:
\begin{enumerate}[label=\textbf{(C\arabic*)}]
\item\label{def:feasible:partition} The routes $R_1, \ldots, R_H$ partition the customer set $V_C$. In other words, each customer is visited on exactly one route.

\item\label{def:feasible:fleet_size} The number of routes performed by vehicles of type $k$ does not exceed their available number; that is, $\sum_{h = 1}^H \mathbb{I}[\kappa_h = k] \leq m_k$ for all $k \in K$, where $\mathbb{I}[\mathcal{E}]$ is the indicator function that evaluates to $1$ if the expression $\mathcal{E}$ is true and $0$ otherwise.

\item\label{def:feasible:capacity} The capacities of all vehicles are respected; that is, $\sum_{i \in R_h} q_i \leq Q_{\kappa_h}$ for all $h \in \{1, \ldots, H\}$.
\end{enumerate}
The cost of a feasible solution is defined to be the sum of the costs of its individual routes, $c\hvrpSol = \sum_{h = 1}^H c(R_h, \kappa_h)$. 
The goal of the HVRP is then to determine a feasible solution of minimum cost.

The HVRP model generalizes several VRP variants that have been studied in the literature~\cite{Baldacci2009:MP}.
The various characteristics that distinguish these variants are summarized in Table~\ref{table:summary_of_HVRP_variants}.
\begin{enumerate}
\item The classical Capacitated VRP (CVRP), in which a homogeneous fleet of $v$ vehicles of identical capacity $Q$ are available at a central depot. The HVRP reduces to the CVRP, if we set $m = 1$, $m_1 = v$, $Q_1 = Q$ and $f_1 = 0$ (i.e., there are no fixed costs associated with the vehicles).

\item The HVRP with no fixed costs ($f_k = 0$ for all $k \in K$) but with vehicle-dependent routing costs, commonly denoted as the HVRPD.

\item The Fleet Size and Mix VRP with fixed costs, vehicle-dependent routing costs, and in which an unlimited number of vehicles of each type is available ($m_k = n$ for all $k \in K$), commonly denoted as FSMFD.

\item The Fleet Size and Mix VRP with fixed costs, vehicle-independent routing costs ($c_{ijk_1} = c_{ijk_2}$ for all $(k_1, k_2) \in K \times K$ and $(i, j) \in E$) and in which an unlimited number of vehicles of each type is available ($m_k = n$ for all $k \in K$), commonly denoted as FSMF.

\item The Fleet Size and Mix VRP with no fixed costs ($f_k = 0$ for all $k \in K$), vehicle-dependent routing costs, and in which an unlimited number of vehicles of each type is available ($m_k = n$ for all $k \in K$), commonly denoted as FSMD.

\item The Site Dependent VRP (SDVRP), in which each customer site $i \in V_C$ can only be visited by a subset of vehicle types $K_i \subseteq K$, representing site-specific constraints. There are no fixed costs ($f_k = 0$ for all $k \in K$) and the routing costs are vehicle-independent but site-dependent. In other words, the routing costs are defined as follows (where we have defined $K_0 = K$):
\[
c_{ijk} = \begin{cases}
\hat{c}_{ij}  & \text{ if } k \in K_i \cap K_j,\\
+\infty & \text{ otherwise,}
\end{cases}\;\;\; \forall (i, j) \in E.
\]

\item The Multi-Depot CVRP (MDVRP), in which a homogeneous fleet of vehicles are stationed at $m$ distinct depots. The vehicle capacities are identical ($Q_k = Q$ for all $k \in K$), their number is unlimited ($m_k = n$ for all $k \in K$) and there are no fixed costs associated with their use ($f_k = 0$ for all $k \in K$). The routing costs are vehicle-independent and are represented by a $(n + m) \times (n + m)$ cost matrix $\hat{c}$, where $\hat{c}_{n+k\, j}$ represents the routing cost along the edge connecting the $k^\text{th}$ depot and customer $j \in V_C$, for all $k \in K$. In other words, we have:
\[
c_{ijk} = \begin{cases}
\hat{c}_{n+k\,j}  & \text{ if } i = 0,\\
\hat{c}_{ij} & \text{ otherwise,}
\end{cases}\;\;\; \forall (i, j) \in E, \; \forall k \in K.
\]
\end{enumerate}

\begin{table}
\centering
\caption{Distinguishing characteristics of the problem variants studied in the literature.}
\begin{tabularx}{\textwidth}{lCCCC}
\toprule
Problem variant & Vehicle fleet & Fleet size & Fixed costs & Routing costs \\
\midrule
CVRP    & Homogeneous   & Limited    & Ignored          & Independent \\
HVRPFD  & Heterogeneous & Limited    & Considered       & Dependent \\
HVRPD   & Heterogeneous & Limited    & Ignored          & Dependent \\
FSMFD   & Heterogeneous & Unlimited  & Considered       & Dependent \\
FSMD    & Heterogeneous & Unlimited  & Considered       & Independent \\
FSMF    & Heterogeneous & Unlimited  & Ignored          & Dependent \\
SDVRP   & Heterogeneous & Limited    & Ignored          & Dependent \\
MDVRP   & Homogeneous   & Unlimited  & Ignored          & Dependent \\
\bottomrule
\end{tabularx}
\label{table:summary_of_HVRP_variants}
\end{table}

In practice, it is often the case that the customer demands are not known precisely when the vehicles are to be dispatched.
In such cases, one option is to replace the unknown demands by their `nominal' values (e.g., by considering a historical sample average) and solve the original deterministic model.
Unfortunately, this would often lead to situations in which the constructed vehicle routes `fail' during their execution (e.g., the vehicles might exceed their carrying capacity in a pickup problem or fail to deliver the demanded quantity in a delivery problem), particularly in situations that deviate from the nominal.
To prevent the occurrence of such situations, we adopt a robust optimization approach and assume that the customer demands can take \emph{any} values from an \emph{uncertainty set} $\mathcal{Q} \subseteq \mathbb{R}^n_{+}$.
This uncertainty set should be chosen in a way that reflects the decision-maker's \emph{a priori} confidence regarding the possible values that the customer demands may take.
Whenever historical data is available, as is the case in practice, statistical models can be used to explicitly parameterize the size and shape of the uncertainty set over this \textit{a priori} confidence level.
We shall revisit this matter in Section~\ref{sec:wc_evaluation}.
For now, we shall only assume, without loss of generality, that the uncertainty set is a non-empty, closed and bounded subset of $\mathbb{R}^n_{+}$.

For a given choice of the uncertainty set $\mathcal{Q}$, the goal in the \emph{robust HVRP} is to determine a \emph{robust feasible} solution of minimum cost.
The solution $\hvrpSol$ is said to be robust feasible if and only if it satisfies conditions~\ref{def:feasible:partition},~\ref{def:feasible:fleet_size} and~\ref{def:feasible:hvrp:robust_capacity}.
\begin{enumerate}[label=\textbf{(C\arabic*r)}]
\setcounter{enumi}{2}
\item\label{def:feasible:hvrp:robust_capacity} The capacities of all vehicles are respected under any realization of the customer demands from the uncertainty set; that is, $\sum_{i \in R_h} q_i \leq Q_{\kappa_h}$ for all $h \in \{1, \ldots, H\}$ and all $q \in \mathcal{Q}$.
\end{enumerate}
By construction, the condition~\ref{def:feasible:hvrp:robust_capacity} is equivalent to verifying whether $\max\limits_{q \in \mathcal{Q}} \sum_{i \in R_h} q_i \leq Q_{\kappa_h}$ is satisfied for each route $h \in \{1, \ldots, H\}$.
In other words, the total load carried by any vehicle must be less than its capacity under the worst-case realization of the customer demands from the uncertainty set $\mathcal{Q}$.
Efficiently evaluating the worst-case load, and hence, verifying condition~\ref{def:feasible:hvrp:robust_capacity}, is key to developing efficient (exact and heuristic) algorithms for the robust HVRP.
The efficient computation of the worst-case load is the subject of the next section, whereas its integration into heuristic and exact methods is discussed in Sections~\ref{sec:local_search} and~\ref{sec:milp_and_bnc}, respectively.

\section{Efficient Computation of the Worst-Case Load}\label{sec:wc_evaluation}
Given a vehicle route $R$, we define its \emph{worst-case load} as the optimal value of the following problem:
\begin{equation}\label{eq:worst_case_problem}
\max_{q \in \mathcal{Q}} \sum_{i \in R} q_i.
\end{equation}
First, note that we can assume, without loss of generality, that the uncertainty set $\mathcal{Q}$ is convex. This is because the objective function of problem~\eqref{eq:worst_case_problem} is linear and therefore, we can equivalently replace the feasible region $\mathcal{Q}$ with its convex hull.
Therefore, in general, computing the worst-case load of a route requires the solution of a convex optimization problem.
Even if $\mathcal{Q}$ is a polyhedron, say in $N$ dimensions and described by $M$ inequalities, then computing an optimal solution of problem~\eqref{eq:worst_case_problem} requires $\mathcal{O}(MN^2)$ arithmetic operations using interior point methods~\cite{BoydVandenberghe:convex_optimization}.

As we shall see later in Sections~\ref{sec:local_search} and~\ref{sec:milp_and_bnc}, this is extremely slow when used in the context of a heuristic or an exact method where hundreds or even thousands of routes are constructed \emph{every second}, and their associated worst-case loads must be computed in order to verify if condition~\ref{def:feasible:hvrp:robust_capacity} is satisfied.
It would be impractical to call a general-purpose convex optimization solver to solve problem~\eqref{eq:worst_case_problem} in such cases.
Furthermore, the successive routes constructed by a heuristic method differ only marginally; therefore, if we know the worst-case load of a route $R$, then we would like to compute the worst-case load of routes $R'$ which are ``almost similar'' to $R$ with minimal additional effort.
It is not clear how this can be achieved using a general-purpose solver.

Fortunately, it can be shown that the worst-case load can be efficiently computed for a broad class of popular but practically-relevant uncertainty sets. Section~\ref{sec:wc_evaluation:uncertainty_sets} elaborates on the structure of these sets, while Section~\ref{sec:wc_evaluation:closed_form_expressions} illustrates how the worst-case load can be efficiently in such cases.

\subsection{Uncertainty Sets}\label{sec:wc_evaluation:uncertainty_sets}
We note that if the uncertainty set is rectangular; that is, $\mathcal{Q} = \{q_1 : q \in \mathcal{Q} \} \times \ldots \times \{q_n : q \in \mathcal{Q}\}$, then the optimal solution of problem~\eqref{eq:worst_case_problem} is attained when each customer demand attains its worst realization individually (defined as $\obar{q}_i = \max\{q_i : q \in \mathcal{Q}\}$), irrespective of the other customer demands.
However, this is a very conservative choice as it is unlikely that all customer demands will simultaneously attain their maximum possible values.
Therefore, we only consider those cases where the uncertainty set is not rectangular.

\subsubsection{Budget sets}\label{sec:wc_evaluation:uncertainty_sets:budgets}
Consider the uncertainty set of the following form (see also Figure~\ref{figure:set_budget}):
\begin{equation}\label{eq:set_disjoint_budget}
\mathcal{Q}_B = \left\{
q \in [\ubar{q}, \obar{q}]: \sum_{i \in B_l} q_i \leq b_l \;\;\; \forall l \in \{1, \ldots, L\}.
\right\}
\end{equation}
This uncertainty set is formed by intersecting the $n$-dimensional hyperrectangle $[\ubar{q}, \obar{q}]$ with the $L \in \mathbb{N}$ budget constraints involving customer subsets $B_l \subseteq V_C$.
The $l^\text{th}$ budget constraint imposes a limit $b_l \in \mathbb{R}_{+}$ on the cumulative demand of the customers in the set $B_l$.
Observe that, by setting $b_l = \sum_{i \in B_l} \ubar{q}_i$, for all $l \in \{1, \ldots, L\}$, the uncertainty set reduces to a singleton $\mathcal{Q}_B = \{\ubar{q}\}$, whereas by setting $b_l = \sum_{i \in B_l} \obar{q}_i$, the uncertainty set becomes the $n$-dimensional hyperrectangle $[\ubar{q}, \obar{q}]$.
To exclude empty sets, we shall assume, without loss of generality, that $\ubar{q} \leq \obar{q}$ and $\sum_{i \in B_l} \ubar{q}_i \leq b_l$.

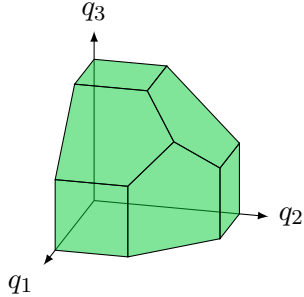
\begin{figure}[!htb]
\centering
\tdplotsetmaincoords{70}{105}
\begin{tikzpicture}[scale=2.0,tdplot_main_coords,face/.style={fill opacity=0.5,fill=green!80!blue}]
\draw[-latex] (0,0,0) -- (1.3,0,0) node[anchor=north east]{$q_1$};
\draw[-latex] (0,0,0) -- (0,1.2,0) node[anchor=west]{$q_2$};
\draw[-latex] (0,0,0) -- (0,0,1.2) node[anchor=south]{$q_3$};

\coordinate (O) at (0,0,0);
\coordinate (D) at (0.75,0.75,0.75);
\coordinate (A1) at (0,0,1);
\coordinate (A2) at (0.5,0,1);
\coordinate (A3) at (0.5,0.5,1);
\coordinate (A4) at (0,0.5,1);
\coordinate (B1) at (1,0,0);
\coordinate (B2) at (1,0.5,0);
\coordinate (B3) at (1,0.5,0.5);
\coordinate (B4) at (1,0,0.5);
\coordinate (C1) at (0,1,0);
\coordinate (C2) at (0.5,1,0);
\coordinate (C3) at (0.5,1,0.5);
\coordinate (C4) at (0,1,0.5);

\draw [face] (A1) -- (A2) -- (A3) -- (A4) -- cycle;
\draw [face] (B1) -- (B2) -- (B3) -- (B4) -- cycle;
\draw [face] (C1) -- (C2) -- (C3) -- (C4) -- cycle;

\draw [face] (A2) -- (A3) -- (D) -- (B3) -- (B4) -- cycle;
\draw [face] (A3) -- (A4) -- (C4) -- (C3) -- (D) -- cycle;
\draw [face] (B2) -- (B3) -- (D) -- (C3) -- (C2) -- cycle;
\end{tikzpicture}
\caption{Example of a budget uncertainty set with $L = 3$ budget constraints.}
\label{figure:set_budget}
\end{figure}

The budget set $\mathcal{Q}_B$ reflects the belief that the customer demand $q_i$ can individually vary between $\ubar{q}_i$ and $\obar{q}_i$, but the cumulative customer demand over various subsets cannot exceed a certain limit.
Intuitively, this belief is rooted in the fact that unless the customer demands exhibit perfect correlations, it is unlikely that they will attain their maximum values simultaneously.
Statistically, budget sets are motivated from limit laws of probability, such as the \emph{central limit theorem}.
Indeed, if the customer demands are independent random variables with means $q^0_i$ and variances $\sigma^2_i$, for $i \in V_C$, then under mild technical conditions, the Lyapunov central limit theorem implies that for sufficiently large $\abs{B_l}$, the sum of the normalized customer demands in the set $B_l$ converges in distribution to a standard normal random variable.
Stated differently, the inequality
\begin{equation*}
\sum_{i \in B_l} q_i \leq \sum_{i \in B_l} q^0_i +  \Phi^{-1}(\gamma)s_l, \;\; \text{where } s_l^2 = \sum_{i \in B_l} \sigma_i^2,
\end{equation*}
is satisfied with probability $\gamma \in (0, 1)$, where $\Phi^{-1}(\cdot)$ denotes the inverse cumulative distribution function of the standard normal random variable.
One can verify that this inequality can be incorporated as a budget constraint in the uncertainty set $\mathcal{Q}_B$.
Thus, we can use standard statistical tools to estimate $q^0_i$ and $\sigma$, and control the size of the uncertainty set using the probability level $\gamma$.
The shape of the uncertainty set can be controlled by selecting appropriate customer subsets $B_l$.
For example, these could represent geographical regions such as municipalities, counties or states.

In the most general case, computing the worst-case load~\eqref{eq:worst_case_problem} for the budget uncertainty set amounts to solving a \emph{fractional packing problem}, for which a $(1+\epsilon)$-approximate solution can be computed in $\mathcal{O}(\epsilon^{-2}Ln)$ time, assuming $L \geq n$~\cite{Young2001}.
A much more efficient computation of the worst-case load is possible if we make the additional assumption that the customer subsets $B_l$ are pairwise \emph{disjoint}, that is, $B_l \cap B_{l'} = \emptyset$ for all $l \neq l'$.
In the remainder of the paper, we shall therefore assume that this additional requirement is satisfied.

\subsubsection{Factor models}
Consider the uncertainty set of the following form (see also Figure~\ref{figure:set_factor_model}):
\begin{equation}\label{eq:set_factor_model}
\mathcal{Q}_F = \left\{
q \in \mathbb{R}^n : q = q^0 + \Psi \xi \text{ for some } \xi \in \Xi_F
\right\}, \text{where }
\Xi_F = \left\{\xi \in [-1, 1]^F: \big\lvert{e^\top \xi}\big\rvert \leq \beta F \right\}.
\end{equation}
Here, $q^0 \in \mathbb{R}^n_{+}$, $F \in \mathbb{N}$, $\Psi \in \mathbb{R}^{n \times F}$ and $\beta \in [0, 1]$ are parameters that need to be specified by the modeler. Note that $e \in \mathbb{R}^F$ denotes the vector of ones.

\begin{figure}[!htb]
\centering
\tdplotsetmaincoords{70}{-110} %-105?
\begin{tikzpicture}[scale=2.5,%
                    tdplot_main_coords,%
                    face/.style={fill opacity=0.5,fill=green!80!blue},%
                    scenario/.style={circle,inner sep=1pt,draw=red!80!blue,fill=red!80!blue}]
\draw[-latex] (1.5,1.4,0.75) -- (0.5,1.4,0.75) node[anchor=north west]{$q_1$};
\draw[-latex] (1.5,1.4,0.75) -- (1.5,0.0,0.75) node[anchor=west]{$q_2$};
\draw[-latex] (1.5,1.4,0.75) -- (1.5,1.4,1.50) node[anchor=south]{$q_3$};

\coordinate (O) at (1,1,1);
\coordinate (A1) at ( 1.4 , 1.6 , 1  ); %
\coordinate (A2) at ( 1.52, 1.68, 1.2); %
\coordinate (A3) at ( 0.88, 0.72, 1.2);
\coordinate (A4) at ( 0.6 , 0.4 , 1  );
\coordinate (A5) at ( 0.48, 0.32, 0.8);
\coordinate (A6) at ( 1.12, 1.28, 0.8);

\draw [face] (A1) -- (A2) -- (A3) -- (A4) -- (A5) -- (A6) -- cycle;
\node [scenario] at (O) {};
\node [above right = -1.35mm and 0.1mm of O,color=red!80!blue] (O) {$q^0$};
\end{tikzpicture}
\caption{Example of a factor model uncertainty set with $F = 2$ factors.}
\label{figure:set_factor_model}
\end{figure}
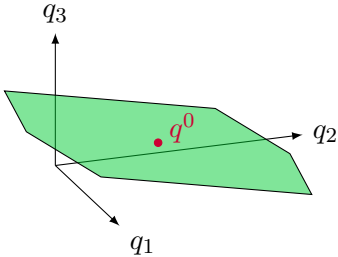

The uncertainty set~\eqref{eq:set_factor_model} stipulates that the unknown customer demands $q$ are distributed around a nominal demand vector $q^0$, subject to an additive disturbance of $\Psi \xi$.
This disturbance is a linear combination of independent factors $\xi_1, \ldots, \xi_F$ that reside in the $F$-dimensional hypercube.
Typically, one has $F \ll n$, so that the linear operator $\Psi$ allows us to model correlations in the (possibly high dimensional space of) customer demands through correlations in the low-dimensional space of the factors.
The matrix $\Psi$ is also known as the \emph{factor loading matrix} and whenever historical data is available, it can be constructed using statistical tools such as principal components analysis or factor analysis.
The constraint $\big\lvert{e^\top \xi}\big\rvert \leq \beta F$ reflects the belief that not all of the independent factors can simultaneously attain their extreme values.
For example, setting $\beta = 0$ will enforce that as many factors will be above $0$ as there will be below $0$, and the resulting factor model has also been referred to as ``zero-net-alpha adjustment'' in portfolio optimization~\cite{Ceria2006}.
Similarly, setting $\beta = 1$ will reduce $\Xi_F$ to an $F$-dimensional hypercube.
In general, whenever historical data is available, an appropriate value of $\beta$ can be chosen by combining a central limit law or tail bound with an \emph{a priori} confidence level, as discussed in Section~\ref{sec:wc_evaluation:uncertainty_sets:budgets}.

\subsubsection{Ellipsoidal sets}
Consider the uncertainty set of the following form (see also Figure~\ref{figure:set_ellipsoid}):
\begin{equation}\label{eq:set_ellipsoid}
\mathcal{Q}_E = \left\{
q \in \mathbb{R}^n : q = q^0 + \Sigma^{1/2} \xi \text{ for some } \xi \in \Xi_E
\right\}, \text{where }
\Xi_E = \left\{\xi \in \mathbb{R}^n: \xi^\top \xi \leq 1 \right\}.
\end{equation}
Here, $q^0 \in \mathbb{R}^n_{+}$ and $\Sigma \in \mathbb{S}^n_{+}$ are parameters that need to be specified by the modeler, and $\mathbb{S}^n_{+}$ denotes the cone of symmetric positive semidefinite matrices.
Whenever $\Sigma$ is a diagonal matrix, we shall refer to the resulting uncertainty set as an \emph{axis-parallel} ellipsoidal set.

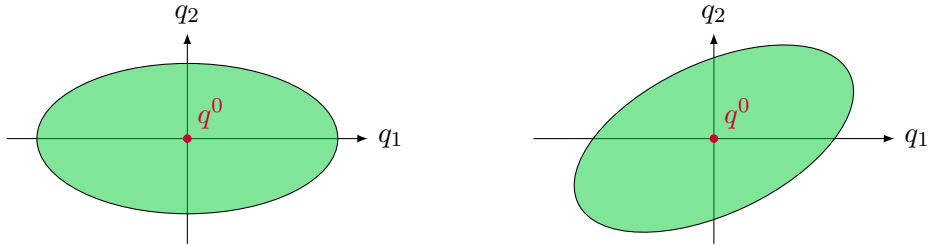
\begin{figure}[!htb]
\centering
\tdplotsetmaincoords{70}{-105}
\begin{tikzpicture}[face/.style={fill opacity=0.5,fill=green!80!blue},%
                    scenario/.style={circle,inner sep=1pt,draw=red!80!blue,fill=red!80!blue}]
\begin{scope}[scale=2.0]
\draw[-latex] (-1.2,0) -- (1.2,0) node[anchor=west]{$q_1$};
\draw[-latex] (0,-0.7) -- (0,0.7) node[anchor=south]{$q_2$};
\draw [face] (0,0) ellipse (1 and 0.5);
\node [scenario] at (0, 0) {};
\node [above right,color=red!80!blue] at (0,0) {$q^0$};
\end{scope}
\begin{scope}[scale=2.0,xshift=3.5cm,rotate=25.4] %25.4
\draw[-latex] (-1.0841,0.5146) -- (1.0841,-0.5146) node[anchor=west]{$q_1$};
\draw[-latex] (-0.3,-0.632) -- (0.3,0.632) node[anchor=south]{$q_2$};
\draw [face] (0,0) ellipse (1 and 0.5);
\node [scenario] at (0, 0) {};
\node [above right,color=red!80!blue] at (0,0) {$q^0$};
\end{scope}
\end{tikzpicture}
\caption{Example of an axis-parallel (left) and general (right) ellipsoidal set.}
\label{figure:set_ellipsoid}
\end{figure}

The uncertainty set~\eqref{eq:set_ellipsoid} stipulates that the customer demands can only attain values in an ellipsoid centered around a nominal demand vector $q^0$.
We note that if the matrix $\Sigma$ is singular, then this ellipsoid is degenerate.
Otherwise, the set $\mathcal{Q}_E$ can be equivalently represented as follows:
\begin{equation*}
\mathcal{Q}_E = \left\{
q \in \mathbb{R}^n : \left(q - q^0 \right)^\top \Sigma^{-1} \left(q - q^0 \right) \leq 1
\right\}.
\end{equation*}
This uncertainty set reflects the belief that the customer demand vector $q$ is a multivariate normal random variable with (unknown) mean $\mu$ and (unknown) covariance matrix $\Sigma^\text{true}$.
The resulting ellipsoid can therefore be identified as a confidence region for the unknown mean $\mu$.
Stated differently, suppose $D > n$ records of historical data are available, and $q^0$ and $\Sigma$ denote the associated sample mean and (unbiased) sample covariance, respectively.
Then, the inequality,
\begin{equation*}
\left(\mu - q^0\right)^\top \Sigma^{-1} (\mu - q^0) \leq \frac{n}{D - n}\frac{D - 1}{D} H_{n;D-n}^{-1}(\gamma),
\end{equation*}
is satisfied with probability $\gamma \in (0, 1)$, where $H_{n;D-n}^{-1}$ denotes the inverse cumulative distribution function of the $F$-distribution with parameters $n$ and $D - n$.
Thus, we can control the size of the ellipsoidal uncertainty set using the probability level $\gamma$.

\subsubsection{Cardinality-constrained sets}
Consider the uncertainty set of the form (see also Figure~\ref{figure:set_cardinality_constrained}):
\begin{equation}\label{eq:set_cardinality_constrained}
\mathcal{Q}_G = \left\{
q \in [q^0, q^0 + \hat{q}] : q = q^0 + (\hat{q} \circ \xi) \text{ for some } \xi \in \Xi_G
\right\}, \text{where }
\Xi_G = \left\{\xi \in [0, 1]^n: e^\top \xi \leq \Gamma \right\}.
\end{equation}
Here, $q^0 \in \mathbb{R}^n_{+}$, $\hat{q} \in \mathbb{R}^n_{+}$ and $\Gamma \in [0, n]$ are parameters that need to be specified by the modeler. Note that $e \in \mathbb{R}^n$ denotes the vector of ones and $(\hat{q} \circ \xi) \in \mathbb{R}^n$ denotes the Hadamard product between vectors $\hat{q}$ and $\xi$; that is, $(\hat{q} \circ \xi)_i = \hat{q}_i \xi_i$ for all $i = 1, \ldots, n$.

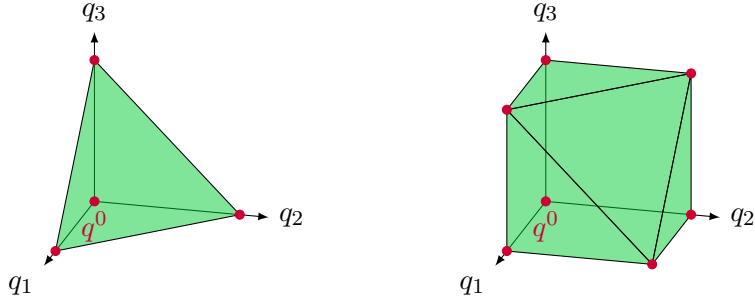
\begin{figure}[!htb]
\centering
\tdplotsetmaincoords{70}{105}
\begin{tikzpicture}[face/.style={fill opacity=0.5,fill=green!80!blue},%
                    scenario/.style={circle,inner sep=1.2pt,draw=red!80!blue,fill=red!80!blue}]
\begin{scope}[scale=2.0,tdplot_main_coords]
\draw[-latex] (0,0,0) -- (1.3,0,0) node[anchor=north east]{$q_1$};
\draw[-latex] (0,0,0) -- (0,1.2,0) node[anchor=west]{$q_2$};
\draw[-latex] (0,0,0) -- (0,0,1.2) node[anchor=south]{$q_3$};

\coordinate (O) at (0,0,0);
\coordinate (A1) at (0,0,1);
\coordinate (B1) at (1,0,0);
\coordinate (C1) at (0,1,0);

\draw [face] (A1) -- (B1) -- (C1) -- cycle;

\node [scenario] at (O) {};
\node [scenario] at (A1) {};
\node [scenario] at (B1) {};
\node [scenario] at (C1) {};
\node [below,color=red!80!blue] at (0,0) {$q^0$};
\end{scope}
\begin{scope}[scale=2.0,tdplot_main_coords,xshift=3cm]
\draw[-latex] (0,0,0) -- (1.3,0,0) node[anchor=north east]{$q_1$};
\draw[-latex] (0,0,0) -- (0,1.2,0) node[anchor=west]{$q_2$};
\draw[-latex] (0,0,0) -- (0,0,1.2) node[anchor=south]{$q_3$};

\coordinate (O) at (0,0,0);
\coordinate (A1) at (0,0,1);
\coordinate (B1) at (1,0,0);
\coordinate (C1) at (0,1,0);
\coordinate (AB) at (1,0,1);
\coordinate (BC) at (1,1,0);
\coordinate (CA) at (0,1,1);

\draw [face] (AB) -- (BC) -- (CA) -- cycle;
\draw [face] (A1) -- (AB) -- (CA) -- cycle;
\draw [face] (B1) -- (AB) -- (BC) -- cycle;
\draw [face] (C1) -- (CA) -- (BC) -- cycle;

\node [scenario] at (O) {};
\node [scenario] at (A1) {};
\node [scenario] at (B1) {};
\node [scenario] at (C1) {};
\node [scenario] at (AB) {};
\node [scenario] at (BC) {};
\node [scenario] at (CA) {};
\node [below,color=red!80!blue] at (0,0) {$q^0$};
\end{scope}
\end{tikzpicture}
\caption{Example of a cardinality-constrained set with $\Gamma = 1$ (left) and $\Gamma = 2$ (right).}
\label{figure:set_cardinality_constrained}
\end{figure}

The uncertainty set stipulates that each customer demand $q_i$ can deviate from its nominal value $q^0_i$ by up to $\hat{q}_i$.
However, the total number of demands that can simultaneously deviate from their nominal values is bounded by $\ceil{\Gamma}$; of these, $\floor{\Gamma}$ customer demands can maximally deviate up to $\hat{q}$ while one customer demand can deviate up to $(\Gamma - \floor{\Gamma})\hat{q}_i$.
For example, if we set $\Gamma = 0$, then the uncertainty set reduces to a singleton $\mathcal{Q}_G = \{q^0\}$, whereas if we set $\Gamma = n$, then the uncertainty set becomes the $n$-dimensional hyperrectangle $[q^0, q^0 + \hat{q}]$.
Observe that $\Xi_G$ is the convex hull of the set $\big\{\xi \in [0, 1]^n : \lVert \xi \rVert_0 \leq \Gamma \big\}$, where $\lVert \xi \rVert_0$ counts the number of non-zero elements in $\xi$.
Therefore, the inequality $e^\top \xi \leq \Gamma$ may be interpreted as constraining the number of elements which may simultaneously deviate from their nominal values, which explains the name ``cardinality-constrained''.
This uncertainty set was originally proposed in~\cite{Bertsimas2004:price_of_robustness} and is also popularly referred to as a ``budgeted'' or ``gamma'' uncertainty set.

Using a similar argument as in~\cite{Bertsimas2004:price_of_robustness}, it can be shown that if the demand of each customer $i \in V_C$ is a symmetric and bounded random variable $\tilde{q}_i$ that takes values in $\left[q^0_i - \hat{q}_i, q^0_i + \hat{q}_i\right]$, then the \emph{actual} worst-case load of any route $R$ satisfies
\begin{equation*}
\text{Probability}\left(\sum_{i \in R} \tilde{q}_i \leq \max_{q \in \mathcal{Q}_G} \sum_{i \in R} q_i \right) \geq 1 - \exp(-\Gamma^2/2n).
\end{equation*}
In other words, if up to $\Gamma$ customer demands actually deviate from their nominal values, then the actual worst-case load of route $R$ is bounded by the quantity $\max\limits_{q \in \mathcal{Q}_G} \sum_{i \in R} q_i$, whereas even if more than $\Gamma$ customer demands deviate from their nominal values, then the actual worst-case load is bounded with very high probability.
We note that a tighter probabilistic bound has been established in~\cite{Bertsimas2004:price_of_robustness}.
In practice, one can use any of these bounds to determine a suitable value of $\Gamma$.

\subsubsection{Discrete sets}
Consider the uncertainty set of the following form (see also Figure~\ref{figure:set_discrete}):
\begin{equation}\label{eq:set_discrete}
\mathcal{Q}_D = \mathop{\text{conv}}\left\{
q^{(j)} : j = 1,\ldots,D
\right\},
\end{equation}
where $\mathop{\text{conv}}(\cdot)$ denotes the convex hull of a finite set of points.
Here, $q^{(1)}, \ldots q^{(D)} \in \mathbb{R}^n_{+}$ are $D \in \mathbb{N}$ distinct realizations of the uncertain customer demands that need to be specified by the modeler.

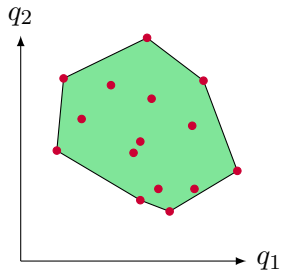
\begin{figure}[!htb]
\centering
\begin{tikzpicture}[scale=3.0,%
                    face/.style={fill opacity=0.5,fill=green!80!blue},%
                    scenario/.style={circle,inner sep=1pt,draw=red!80!blue,fill=red!80!blue}]
\draw[-latex] (0,0) -- (1,0) node[anchor=west]{$q_1$};
\draw[-latex] (0,0) -- (0,1) node[anchor=south]{$q_2$};

\coordinate (P0) at (0.53,0.27);
\coordinate (P1) at (0.53,0.53);
\coordinate (P2) at (0.16,0.49);
\coordinate (P3) at (0.19,0.81);
\coordinate (P4) at (0.66,0.22);
\coordinate (P5) at (0.27,0.63);
\coordinate (P6) at (0.56,0.99);
\coordinate (P7) at (0.50,0.48);
\coordinate (P8) at (0.81,0.80);
\coordinate (P9) at (0.96,0.40);
\coordinate (P10) at (0.76,0.60);
\coordinate (P11) at (0.40,0.78);
\coordinate (P12) at (0.61,0.32);
\coordinate (P13) at (0.58,0.72);
\coordinate (P14) at (0.77,0.32);

\draw [face] (P6) -- (P3) -- (P2) -- (P0) -- (P4) -- (P9) -- (P8) -- cycle;

\node[scenario] at (P0) {};
\node[scenario] at (P1) {};
\node[scenario] at (P2) {};
\node[scenario] at (P3) {};
\node[scenario] at (P4) {};
\node[scenario] at (P5) {};
\node[scenario] at (P6) {};
\node[scenario] at (P7) {};
\node[scenario] at (P8) {};
\node[scenario] at (P9) {};
\node[scenario] at (P10) {};
\node[scenario] at (P11) {};
\node[scenario] at (P12) {};
\node[scenario] at (P13) {};
\node[scenario] at (P14) {};
\end{tikzpicture}
\caption{Example of a discrete uncertainty set with $D = 15$ data points.}
\label{figure:set_discrete}
\end{figure}

The uncertainty set stipulates that the customer demands will take values inside the convex hull of $D$ \emph{a priori} specified demand vectors in $n$-dimensional space.
This set is expected to be meaningful only if sufficient historical records are available; that is, if $D$ is sufficiently large.
However, care must be taken to discard statistical outliers in order to avoid becoming overly risk-averse.
More generally, one may want to consider only a small subset of all historically observed realizations by uniformly sampling a fraction $\gamma \in (0, 1)$ of the available realizations.

\subsection{Closed-Form Expressions of the Worst-Case Load}\label{sec:wc_evaluation:closed_form_expressions}
The key result of this section is that the worst-case load of a vehicle route can be computed in closed-form for each of the five classes of uncertainty sets described in the previous section. This formalized in Proposition~\ref{prop:closed_form_expressions}.
\begin{prop}\label{prop:closed_form_expressions}
Suppose $R$ is a given route and $S \subseteq V_C$ is the set of customers visited on $R$. Then, the worst-case load of route $R$ over the uncertainty sets $\mathcal{Q}_B$, $\mathcal{Q}_F$, $\mathcal{Q}_E$, $\mathcal{Q}_G$ and $\mathcal{Q}_D$ is as follows.
\begingroup
\captionsetup{type=table}
\captionof{table}{Closed-form expressions of the worst-case load of a vehicle route.\label{table:closed_form_expressions}}
\vspace*{-\baselineskip}
\endgroup
\begingroup
\allowdisplaybreaks
\begin{alignat}{2}
\toprule
&\mathcal{Q}   &\qquad& \max_{q \in \mathcal{Q}} \sum_{i \in R} q_i \notag \\ \midrule
&\mathcal{Q}_B && \sum_{i \in S} \obar{q}_i - \sum_{l = 1}^L \max \left\{0, \sum_{i \in S \cap B_l} \big(\obar{q}_i - \ubar{q}_i\big) - \left(b_l - \sum_{i \in B_l} \ubar{q}_i \right)\right\} \label{eq:closed_form_expression_budget}\\
&\mathcal{Q}_F && \sum_{i \in S} q^0_i - \min \left\{\sum_{f = 1}^F \sum_{i \in S} \Psi_{if} - \lambda + \beta F \abs{\lambda} : \lambda \in \left\{0, \sum_{i \in S} \Psi_{if_{\ell^{+}}},\sum_{i \in S} \Psi_{if_{\ell^{-}}} \right\} \right\}, \label{eq:closed_form_expression_factor} \\
&              && \text{\small where $f_1, \ldots f_F$ represents an ordering of the factors such that $\displaystyle\sum_{i \in S} \Psi_{if_1} \geq \ldots \geq \sum_{i \in S} \Psi_{if_F}$}, \notag \\[-2mm]
&              && \text{\small and $\ell^{+} = \ceil{(1 + \beta)F/2}$, $\ell^{-} = \max\{1, \ceil{(1 - \beta)F / 2}\}$.} \notag \\
&\mathcal{Q}_E && \sum_{i \in S} q^0_i + \bigg\lVert \sum_{i \in S} \Sigma^{1/2}_i \bigg\rVert_2, \label{eq:closed_form_expression_ellipsoid} \\
&              && \text{\small where $\Sigma^{1/2}_i$ denotes the $i^\text{th}$ column of $\Sigma^{1/2}$ and $\lVert \cdot \rVert_2$ denotes the $\ell_2$-norm of a vector.} \notag \\
&\mathcal{Q}_G && \sum_{i \in S} q^0_i + \sum_{\ell = 1}^{\min\{\abs{S}, \floor{\Gamma}\}} \hat{q}_{g_{\ell}} + \lambda, \label{eq:closed_form_expression_gamma} \\
&              && \text{\small where $g_1, \ldots g_{\abs{S}}$ represents an ordering of the customers in $S$ such that $\displaystyle \hat{q}_{g_1} \geq \ldots \geq \hat{q}_{g_{\abs{S}}}$}, \notag \\[-2mm]
&              && \text{\small and $\displaystyle \lambda = \left(\Gamma - \floor{\Gamma}\right)\hat{q}_{g_{\floor{\Gamma} + 1}}$, if $\abs{S} \geq \floor{\Gamma} + 1$ and $0$ otherwise.} \notag \\
&\mathcal{Q}_D && \max \left\{\sum_{i \in S} q^{(d)}_i : d \in \{1,\ldots,D\} \right\} \label{eq:closed_form_expression_discrete} \\ \bottomrule
&&&\notag
\end{alignat}
\endgroup
\end{prop}
\vspace*{-\baselineskip}
\begin{proof}
The validity of the expressions for $\mathcal{Q} = \mathcal{Q}_B$ and $\mathcal{Q} = \mathcal{Q}_F$ has been shown in~\cite{Gounaris2013:OR}.

Suppose that $\mathcal{Q} = \mathcal{Q}_E$.
In this case, the worst-case problem~\eqref{eq:worst_case_problem} can be reformulated as follows:
\begin{equation*}
\sum_{i \in S} q^0_i + \max_{\xi \in \mathbb{R}^n} \left\{ \sum_{i \in S} \xi^\top \Sigma^{1/2}_i : \xi^\top \xi \leq 1 \right\}.
\end{equation*}
The above maximization problem is a convex optimization problem that satisfies Slater's constraint qualification (e.g., $0 \in \mathbb{R}^n$ is strictly feasible). Therefore, the Karush-Kuhn-Tucker conditions are both necessary and sufficient to characterize its optimal solution. This leads to expression~\eqref{eq:closed_form_expression_ellipsoid}.

Suppose that $\mathcal{Q} = \mathcal{Q}_G$.
In this case, the worst-case problem~\eqref{eq:worst_case_problem} can be reformulated as follows:
\begin{equation*}
\sum_{i \in S} q^0_i + \max_{\xi \in \mathbb{R}^n_{+}} \left\{ \sum_{i \in S} \hat{q}_i \xi_i : \xi_i \leq 1 \; \forall i = 1, \ldots, n, \;\;\; \sum_{i = 1}^n \xi_i \leq \Gamma \right\}.
\end{equation*}
The above maximization problem is an instance of a \emph{fractional knapsack} problem with $n$ items, each with unit weight; the value of item $i$ is $\hat{q}_i \mathbb{I}[i \in S]$.
It is well known (e.g., see~\cite{Dantzig1957:fractional_knapsack}) that this problem can be solved by a greedy algorithm that considers the items in non-increasing order of their values per unit weight, i.e., in non-increasing order of the $\hat{q}$ values of the items in $S$.
A straightforward application of this greedy algorithm leads to expression~\eqref{eq:closed_form_expression_gamma}.

Suppose now that $\mathcal{Q} = \mathcal{Q}_D$. Since $\mathcal{Q}_D$ is a bounded polytope and the objective function of the worst-case problem~\eqref{eq:worst_case_problem} is linear, the optimal solution is attained at a vertex of $\mathcal{Q}_D$. This leads to expression~\eqref{eq:closed_form_expression_discrete}, since the set of vertices of $\mathcal{Q}_D$ is a subset of the $D$ points that parametrize it.
\end{proof}

Proposition~\ref{prop:closed_form_expressions} suggests that the worst-case load of a route can be computed much faster by evaluating the associated closed-form expressions than by invoking a general-purpose optimization solver.
Furthermore, if we know the worst-case load of a route $R$ that visits a certain subset of customers $S \subseteq V_C$, and we would like to calculate the worst-case load of a route $R'$ that visits exactly one additional or fewer customer $i \in V_C$, i.e., if $R'$ visits $S \cup \{i\}$ or $S \setminus \{i\}$, then the worst-case load can be \emph{incrementally updated} even faster by using appropriate data structures.
This is formalized in Proposition~\ref{prop:cpu_and_memory_requirements}.

\begin{prop}\label{prop:cpu_and_memory_requirements}
The time and storage complexities shown in Table~\ref{table:cpu_and_memory_requirements} can be achieved for computing the worst-case load of a route that visits a set of customers $S$. Here, ``incremental update'' refers to the complexity of updating the worst-case load when a single customer is added to or removed from $S$.
\begin{table}[!h]
\centering
\caption{Time and storage complexities for computing the worst-case load of a vehicle route.}
\begin{tabularx}{0.9\textwidth}{lCCC}
\toprule
\multirow{2}{*}{$\mathcal{Q}$} & From scratch & \multicolumn{2}{c}{Incremental update} \\
\cmidrule(r){2-2} \cmidrule(l){3-4}
& Time & Time & Storage \\
\midrule
$\mathcal{Q}_B$ & $\mathcal{O}(\abs{S})$ & $\mathcal{O}(1)$ & $\mathcal{O}(L)$ \\
$\mathcal{Q}_F$ & $\mathcal{O}(\abs{S}F + F\log F)$ & $\mathcal{O}(F\log F)$ & $\mathcal{O}(F)$ \\
$\mathcal{Q}_E$ (axis-parallel) & $\mathcal{O}(\abs{S})$ & $\mathcal{O}(1)$ & $\mathcal{O}(1)$ \\
$\mathcal{Q}_E$ (general) & $\mathcal{O}(\abs{S}n)$ & $\mathcal{O}(n)$ & $\mathcal{O}(n)$ \\
$\mathcal{Q}_G$ & $\mathcal{O}(\abs{S})$ & $\mathcal{O}(\log(\abs{S}))^\text{a}$ & $\mathcal{O}(n)^\text{b}$ \\
$\mathcal{Q}_D$ & $\mathcal{O}(\abs{S}D)$ & $\mathcal{O}(D)$ & $\mathcal{O}(D)$ \\ \bottomrule
\multicolumn{4}{l}{\footnotesize$^\text{a,b}$These can be improved to $\mathcal{O}(\log (\Gamma))$ and $\mathcal{O}(\Gamma)$, respectively, if only additions to $S$ are considered.}
\end{tabularx}
\label{table:cpu_and_memory_requirements}
\end{table}
\end{prop}
\begin{proof}
\textbullet\ Consider $\mathcal{Q} = \mathcal{Q}_B$. The time complexity for the ``from scratch'' computation follows directly from expression~\eqref{eq:closed_form_expression_budget}.
To enable incremental updates, we use $\mathcal{O}(L)$ storage to keep track of the following quantities: $\pi = \sum_{i \in S} \obar{q}_i$, $\rho_l = \sum_{i \in S \cap B_l} \big(\obar{q}_i - \ubar{q}_i\big) - \left(b_l - \sum_{i \in B_l} \ubar{q}_i \right)$ for all $l = 1, \ldots, L$, and $z = \max_{q \in \mathcal{Q}} \sum_{i \in S} q_i$. Initially, when $S = \emptyset$, we have $(\pi, \rho_l, z) = (0, -(b_l - \sum_{i \in B_l} \ubar{q}_i), 0)$.
Now, suppose that we have updated this data structure to reflect the worst-case load of an ``old'' route visiting the customer set $S$, and that we would like to calculate the worst-case load of a ``new'' route visiting the customer set $S' = S \cup \{j\}$, where customer $j$ participates in the budget $l_j$ (i.e., $j \in B_{l_j}$). For this, we can perform the following update in $\mathcal{O}(1)$ time: 
$\pi^\text{new} \gets \pi^\text{old} + \obar{q}_j$,
$\rho_{l_j}^\text{new} \gets \rho_{l_j}^\text{old} + (\obar{q}_j - \ubar{q}_j)$,
$z^\text{new} \gets z^\text{old} + (\pi^\text{new} - \pi^\text{old}) - ([\rho^\text{new}_{l_j}]_{+} - [\rho^\text{old}_{l_j}]_{+})$, where $[\cdot]_{+} = \max\{\cdot, 0\}$.
If customer $j$ does not participate in any budget, then we would repeat the same steps except $\rho_l$ is not updated.
A similar update applies when $S' = S \setminus \{j\}$ for some $j \in S$.
We do not present this for the sake of brevity.

\textbullet\ Consider $\mathcal{Q} = \mathcal{Q}_F$. The time complexity for the ``from scratch'' computation follows from the $\mathcal{O}(\abs{S})$ time to calculate $\sum_{i \in S} \Psi_{if}$ for each of the $F$ factors and the $\mathcal{O}(F \log F)$ time to compute the ordering $f_1, \ldots, f_F$.
To enable incremental updates, we use $\mathcal{O}(F)$ storage to keep track of the following quantities: $\pi = \sum_{i \in S} {q}^0_i$, $\rho_f = \sum_{i \in S} \Psi_{if}$ for all $f = 1, \ldots, F$, and $z = \max_{q \in \mathcal{Q}} \sum_{i \in S} q_i$.
Initially, when $S = \emptyset$, we have $(\pi, \rho_f, z) = (0, 0, 0)$.
To calculate the worst-case load of a route visiting the customer set $S' = S \cup \{j\}$, we can perform the following update: 
$\pi^\text{new} \gets \pi^\text{old} + q^0_j$,
$\rho_f^\text{new} \gets \rho_f^\text{old} + \Psi_{jf}$ for each $f = 1, \ldots, F$, and
$z^\text{new} \gets \pi^\text{new} + \sum_{g=1}^F \xi^\text{wc}_g \rho^\text{new}_{f_g}$, where $f_1, \ldots, f_F$ is an updated ordering of the factors according to $\rho^\text{new}_{f_1} \geq \ldots \geq  \rho^\text{new}_{f_F}$, and $\xi^\text{wc} \in \Xi_F$ is defined by first computing 
$P = \sum_{f = 1}^F \mathbb{I}[\rho^\text{new}_{f} \geq 0]$, $N = F - P$, $A = \abs{P - N}$, $M = \min\left\{P, N\right\}$, $T = \floor{(A - \floor{\beta A})/2}$ and $Z = \floor{\beta A} + T$, and then consulting the entries of the following table:
\begin{center}
\small
\begin{tabularx}{0.56\textwidth}{c|c|l}
\toprule
\multicolumn{2}{c|}{Case} & $\xi^\text{wc}$ \\
\midrule
\multirow{2}{*}{$P \geq N$} & $T + Z = A$ & $(e^M, e^Z, -e^T, -e^M)$  \\
& $T + Z \neq A$ & $(e^M, e^Z, +\beta F - \floor{\beta F}, -e^T, -e^M)$ \\
\midrule
\multirow{2}{*}{$P < N$} & $T + Z = A$ & $(e^M, e^T, -e^Z, -e^M)$ \\
& $T + Z \neq A$ & $(e^M, e^T, -\beta F + \floor{\beta F}, -e^Z, -e^M)$ \\
\bottomrule
\end{tabularx}
\end{center}
Here, $e^\text{dim}$ denotes the vector of ones in $\mathbb{R}^\text{dim}$. The overall time complexity of the update is $\mathcal{O}(F \log F)$ and is dictated by the sorting operation needed to compute the ordering $f_1, \ldots, f_F$. \footnote{In practice, the addition of a single customer $j$ to a larger set $S$ is unlikely to significantly change the ordering of the factors $f_1, \ldots, f_F$. We can take advantage of this by keeping track of the ordering and making use of a specialized sorting algorithm for ``almost sorted'' arrays, leading to even faster updates.}
%A similar update applies when $S' = S \setminus \{j\}$ for some $j \in S$.

\textbullet\ Consider $\mathcal{Q} = \mathcal{Q}_E$, where $\mathcal{Q}_E$ is axis-parallel; that is, $\Sigma = \text{diag}(\sigma_1^2, \ldots, \sigma_n^2)$ is a diagonal matrix.
In this case, the expression~\eqref{eq:closed_form_expression_ellipsoid} simplifies to $\sum_{i \in S} q^0_i + \sqrt{\sum_{i \in S} \sigma_i^2}$, which can be computed in $\mathcal{O}(\abs{S})$ time.
To enable incremental updates, we use $\mathcal{O}(1)$ storage to keep track of the following quantities: $\pi = \sum_{i \in S} {q}^0_i$, $\rho = \sum_{i \in S} \sigma_i^2$ and $z = \max_{q \in \mathcal{Q}} \sum_{i \in S} q_i$.
Initially, when $S = \emptyset$, we have $(\pi, \rho, z) = (0, 0, 0)$.
To calculate the worst-case load of a route visiting the customer set $S' = S \cup \{j\}$, we can perform the following update in $\mathcal{O}(1)$ time:
$\pi^\text{new} \gets \pi^\text{old} + q^0_j$, 
$\rho^\text{new} \gets \rho^\text{old} + \sigma_j^2$, 
$z^\text{new} \gets \pi^\text{new} + \sqrt{\rho^\text{new}}$.
A similar update applies when $S' = S \setminus \{j\}$ for some $j \in S$.

\textbullet\ Consider now $\mathcal{Q} = \mathcal{Q}_E$, where $\Sigma$ is a general matrix.
In this case, expression~\eqref{eq:closed_form_expression_ellipsoid} can be written as $\sum_{i \in S} q^0_i + \sqrt{\sum_{j = 1}^n \Big(\sum_{i \in S} \Sigma_{ij}^{1/2} \Big)^2 }$, which can be computed in $\mathcal{O}(\abs{S}n)$ time.
To enable incremental updates, we use $\mathcal{O}(n)$ storage to keep track of the following quantities: $\pi = \sum_{i \in S} {q}^0_i$, $\rho_l = \sum_{i \in S} \Sigma_{il}^{1/2}$ for all $l = 1, \ldots, n$, and $z = \max_{q \in \mathcal{Q}} \sum_{i \in S} q_i$.
When $S = \emptyset$, we have $(\pi, \rho_l, z) = (0, 0, 0)$.
To calculate the worst-case load of a route visiting the customer set $S' = S \cup \{j\}$, we can perform the following update in $\mathcal{O}(n)$ time: 
$\pi^\text{new} \gets \pi^\text{old} + q^0_j$, 
$\rho_l^\text{new} \gets \rho_l^\text{old} + \Sigma_{jl}^{1/2}$ for all $l = 1, \ldots, n$, 
$z^\text{new} \gets \pi^\text{new} + \sqrt{\sum_{l = 1}^n \big(\rho_l^{\text{new}}\big)^2}$.
A similar update applies when $S' = S \setminus \{j\}$ for some $j \in S$.

\textbullet\ Consider $\mathcal{Q} = \mathcal{Q}_G$. An examination of expression~\eqref{eq:closed_form_expression_gamma} reveals that we do not need to sort the customers in $S$ with respect to their $\hat{q}$ values. Instead, we only need to identify the subset of $(\floor{\Gamma} + 1)$ customers with the largest $\hat{q}$ values.
This can be achieved in $\mathcal{O}(\abs{S})$ time with a \emph{partition-based selection algorithm}, such as \emph{quickselect} with an appropriate pivoting strategy (e.g., see~\cite{IntroToAlgorithms:2009:3rd_edition}).
To enable incremental updates, we use $\mathcal{O}(n)$ storage to keep track of the following quantities: $\pi = \sum_{i \in S} {q}^0_i$, $s = \min\{\abs{S}, \floor{\Gamma}\}$, array $h^{+} = [\hat{q}_{g_1}, \ldots, \hat{q}_{g_s}]$ implemented as a binary min-heap, array $h_{-} = [\hat{q}_{g_{s + 2}}, \ldots, \hat{q}_{g_{\abs{S}}}]$ implemented as a binary max-heap, $\rho_{+} = $ sum of entries of $h_{+}$, $\rho_{0} = \hat{q}_{g_{s + 1}}$, and $z = \max_{q \in \mathcal{Q}} \sum_{i \in S} q_i$, where we define $h_{-} = \emptyset$, if $\abs{S} \leq \floor{\Gamma} + 1$, and $\rho_0 = 0$, if $s + 1 > \abs{S}$.
Initially, when $S = \emptyset$, we have $(\pi, s, h_{+}, h_{-}, \rho_{+}, \rho_{0}, z) = (0, 0, \emptyset, \emptyset, 0, 0, 0)$.
To calculate the worst-case load of a route visiting the customer set $S' = S \cup \{j\}$, we can perform the following update: 
if $s^\text{old} < \floor{\Gamma}$, then $s^\text{new} \gets s^\text{old} + 1$, $h_{+}^\text{new} \gets h_{+}^\text{old}.\text{insert}(\hat{q}_j)$, $\rho_{+}^\text{new} \gets \rho_{+}^\text{old} + \hat{q}_j$;
else if $\hat{q}_j \leq \rho_{0}$, then $h_{-}^\text{new} \gets h_{-}^\text{old}.\text{insert}(\hat{q}_j)$;
else if $\hat{q}_j < \min (h_{+}^\text{old})$, then $\rho_{0}^\text{new} \gets \hat{q}_j$, $h_{-}^\text{new} \gets h_{-}^\text{old}.\text{insert}(\rho_0^\text{old})$;
else, $h_{+}^\text{new} \gets h_{+}^\text{old}.\text{delete}(\min (h_{+}^\text{old})).\text{insert}(\hat{q}_j)$, $\rho_{+}^\text{new} \gets \rho_{+}^\text{old} + \hat{q}_j - \min (h_{+}^\text{old})$, $\rho_{0}^\text{new} \gets \min (h_{+}^\text{old})$ and $h_{-}^\text{new} \gets h_{-}^\text{old}.\text{insert}(\rho_0^\text{old})$.
In all cases, we also update $\pi^\text{new} \gets \pi^\text{old} + q^0_j$ and $z^\text{new} \gets \pi^\text{new} + \rho_{+}^\text{new} + (\Gamma - \floor{\Gamma}) \rho_0^\text{new}$.
The overall time complexity of the update is $\mathcal{O}(\abs{S})$ and is dictated by the insertion and deletion operations in the heaps $h_{+}$ and $h_{-}$.
Note that $h_{-}$ was not used in determining the value of $z$.
In fact, it is used only when updating the worst-case load after a deletion has occurred.
Consequently, if the quantities are maintained only by the addition of elements into initially empty structures, then $h_{+}$ is sufficient, and since its size is at most $\floor{\Gamma}$, the overall time and storage complexities would be $\mathcal{O}(\log(\Gamma))$ and $\mathcal{O}(\Gamma)$, respectively.

\textbullet\ Consider now $\mathcal{Q} = \mathcal{Q}_D$.
The time complexity for the ``from scratch'' computation follows directly from expression~\eqref{eq:closed_form_expression_discrete}.
To enable incremental updates, we use $\mathcal{O}(D)$ storage to keep track of the following quantities: $\rho_d = \sum_{i \in S} q^{(d)}_i$ for all $d = 1, \ldots, D$, and $z = \max_{q \in \mathcal{Q}} \sum_{i \in S} q_i$.
Initially, when $S = \emptyset$, we have $(\rho_d, z) = (0, 0)$.
To calculate the worst-case load of a route visiting the customer set $S' = S \cup \{j\}$, we can perform the following update in $\mathcal{O}(D)$ time: 
$\rho_d^\text{new} \gets \rho_d^\text{old} + q^{(d)}_j$ for all $d = 1, \ldots, D$, 
$z^\text{new} \gets \max \left\{\rho_d^\text{new} : d = 1, \ldots, D \right\}$.
A similar update applies when $S' = S \setminus \{j\}$ for some $j \in S$.
\end{proof}

\section{Robust Local Search and Metaheuristics}\label{sec:local_search}
The vast majority of metaheuristics for solving large-scale instances of (deterministic) vehicle routing problems are all based on \emph{local search}.
In this section, we illustrate how the results from the previous section allow us to efficiently extend local search to the robust setting.
This robust version of local search can then be incorporated in a modular fashion into any metaheuristic algorithm.
To illustrate this, Appendix~\ref{sec:local_search:ILS} provides details on our implementation of an Iterated Local Search (ILS) metaheuristic, while Appendix~\ref{sec:local_search:AMP} provides details on our implementation of an Adaptive Memory Programming (AMP) metaheuristic. We highlight that our proposed metaheuristic algorithms are simple to implement and adapt, as they introduce few user-defined parameters and they do not incorporate any instance-specific features or spatiotemporal decomposition schemes to accelerate the local search process.

The basis of all local search methods is the repeated use of a set of elementary moves that transform a \emph{current solution} $\hvrpSol$ into a different, \emph{neighbor solution}.
The set of all solutions that can be reached from the current one using a set $Y$ of moves is called the neighborhood of the current solution with respect to the move set, $\Omega_Y\hvrpSol$.
The two major building blocks of all local search methods are therefore, the definition of the \emph{neighborhoods}, and the exploration of the neighborhoods using a \emph{search algorithm}.
The most common neighborhoods are the \emph{node-} and \emph{edge-exchange} neighborhoods that involve the deletion and re-insertion of nodes or edges~\cite{Aarts:1997:LSC:549160,Funke2005} and belong to the family of $k$-Opt and $\lambda$-interchange neighborhoods.
Specific examples include the \emph{relocate}, \emph{exchange} and \emph{2-opt} neighborhoods (see Figure~\ref{figure:move}), with their size $\abs{\Omega_Y\hvrpSol}$, $Y \in \{\text{relocate, exchange, 2-opt}\}$ being $\mathcal{O}(n^2)$.
\begin{figure}[!htb]
\centering
\includegraphics[width=0.45\textwidth]{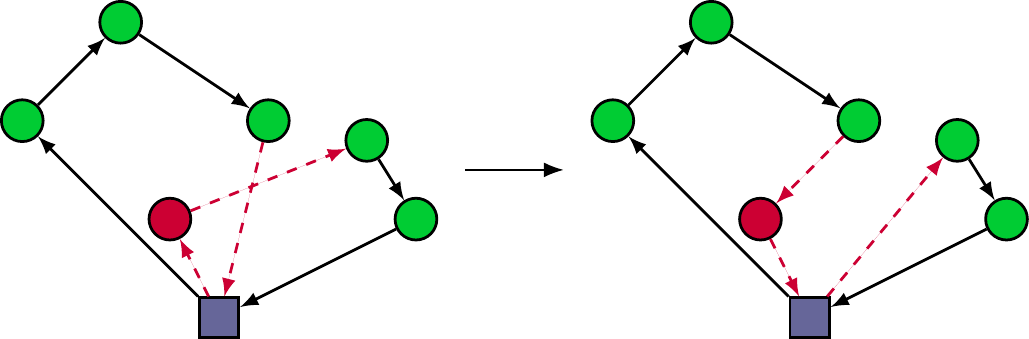}\qquad\quad
\includegraphics[width=0.45\textwidth]{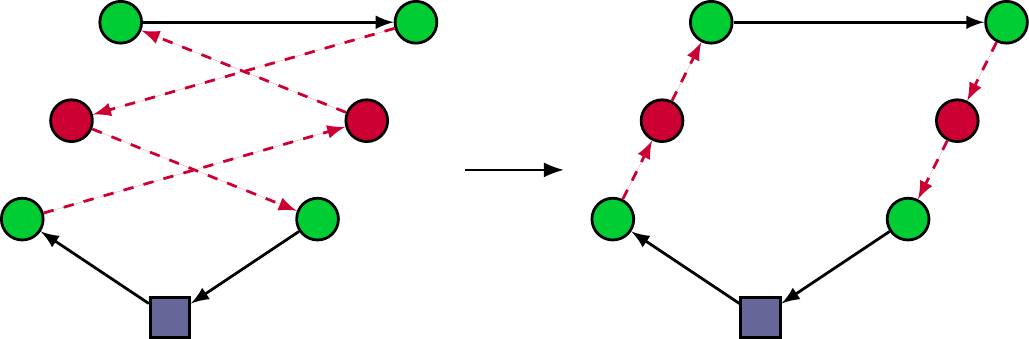}

\medskip

\includegraphics[width=0.45\textwidth]{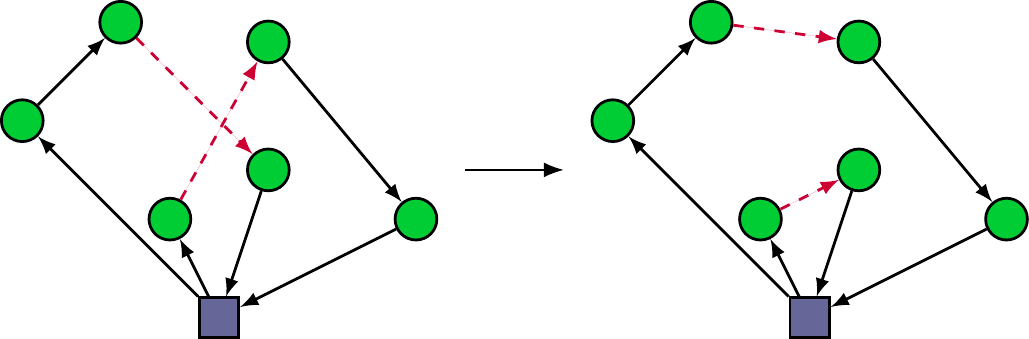}
\caption{Examples of inter-route relocate (top left), intra-route exchange (top right) and inter-route 2-opt (bottom) moves.}
\label{figure:move}
\end{figure}

The goal of a local search algorithm is to efficiently explore the given set of neighborhoods to find \emph{improving neighbor solutions}, i.e., solutions $\hvrpSolx{\prime}$ that are better than the current one $\hvrpSol$. 
To formalize what we mean by ``better'', we shall define the \emph{total penalized cost} of a solution, $\bar{c}\hvrpSol$, as the weighted sum of its transportation costs and worst-case vehicle capacity violations:
\begin{equation}\label{eq:penalized_cost}
\bar{c}\hvrpSol = c\hvrpSol + \varphi^Q p\hvrpSol = c\hvrpSol + \varphi^Q \sum_{h = 1}^H \max \left\{0, \max_{q \in \mathcal{Q}} \sum_{i \in R_h} q_i - Q_{\kappa_h}\right\},
\end{equation}
where $\varphi^Q$ is a large penalty coefficient.
Therefore, a neighbor solution $\hvrpSolx{\prime}$ is said to be improving if $\bar{c}\hvrpSolx{\prime} < \bar{c}\hvrpSol$.
By making $\varphi^Q$ sufficiently large, this is equivalent to lexicographically comparing $p\hvrpSolx{\prime} < p\hvrpSol$ and only if these are equal, then comparing if $c\hvrpSolx{\prime} < c\hvrpSol$.
We note that by doing this, we allow local search to handle also infeasible solutions.
Specifically, if the initial solution to local search (e.g., obtained via a construction heuristic) satisfies the visit constraints~\ref{def:feasible:partition} and fleet availability constraints~\ref{def:feasible:fleet_size} but not necessarily the vehicle capacity constraints~\ref{def:feasible:hvrp:robust_capacity}, then the local search algorithm will first attempt to find a robust feasible neighbor solution.
Once such a neighbor solution is found, the local search does not reenter the infeasible region, thereafter admitting only those neighbor solutions that are feasible.

The above observations imply that even if we restrict our attention to the aforementioned neighborhoods, every iteration of local search involves evaluating the total penalized cost of $\mathcal{O}(n^2)$ neighbor solutions to find at least one improving neighbor solution $\hvrpSolx{\prime}$.
While the evaluation of the transportation cost $c\hvrpSolx{\prime}$ can be done in constant time (by simply updating the current value of $c\hvrpSol$), the evaluation of its robust feasibility, i.e., $p\hvrpSolx{\prime}$, is not trivial, since it amounts to the solution of an inner optimization problem. 
As such, intra-route moves are trivial to evaluate since they only alter the positions of customers within a route, and not the actual set of customers itself, so that $p\hvrpSolx{\prime} = p\hvrpSol$ in such cases.
On the other hand, inter-route moves, which account for the majority of postulated moves, require us to recalculate $p\hvrpSolx{\prime}$.
Fortunately, whenever the uncertainty set $\mathcal{Q}$ is one of the five classes of sets described in Section~\ref{sec:wc_evaluation:uncertainty_sets}, we can employ the data structures described in Proposition~\ref{prop:cpu_and_memory_requirements} to efficiently update the worst-case load of the affected routes (often in constant or sublinear time), and hence efficiently compute $p\hvrpSolx{\prime}$.
Specifically, since each of the aforementioned inter-route moves can be broken down into elementary additions and removals of customers, we can incrementally update the current $p(R, k)$ value of an affected route $R$ to obtain the updated $p(R^\prime, k^\prime)$ value, using the result stated in Proposition~\ref{prop:cpu_and_memory_requirements}.
%We refer to this extension of local search as \emph{robust local search}.
Therefore, our results also generalize to neighborhoods with more complex structure (e.g., cyclical exchanges, ejection chains etc.).

In the context of the HVRP, efficient evaluation of solutions allows us to further generalize the aforementioned neighborhoods to modify also the fleet composition vector $\bm{\kappa}$ in addition to the routes $\mathbf{R}$.
Specifically, whenever the fleet size is unlimited, every inter-route move also modifies the vehicle type $\kappa_{h}$ of each affected route $R_{h}$. The goal is to minimize the worst-case vehicle capacity violation, i.e., $\kappa_h = \arg\min_{k \in K} p(R_h, k)$; if robust feasibility $(p(R_h, k) = 0)$ is possible using at least one vehicle type $k \in K$, then the goal is to minimize the transportation cost, i.e., $\kappa_h = \arg\min_{k \in K} \left\{c(R_h, k) : p(R_h, k) = 0\right\}$.
In other words, the inter-route moves modify the vehicle type $\kappa_h$ to the one that can feasibly perform route $R_h$ under all realizations of the uncertainty at minimum cost.
On the other hand, whenever the fleet size is limited, every inter-route move only exchanges the vehicle types of the affected routes (if doing so can lower $\bar{c}\hvrpSol$).
This ensures that the fleet availability condition~\ref{def:feasible:fleet_size} is never violated.

The above modifications apply to all neighborhoods.
In addition to these, we can also generalize the definition of specific neighborhoods.
For example, prior to exploring the relocate neighborhood, we can add an empty route to the current solution.
The vehicle type of this empty route is one with the largest capacity for which it is possible to do so without violating condition~\ref{def:feasible:fleet_size}.
This allows the relocation of a customer to its own route, thus further expanding the search space.

Finally, we note that it is also possible to generalize the total penalized cost function $\bar{c}$ for specific HVRP variants (see Table~\ref{table:summary_of_HVRP_variants}).
For example, to ensure that site dependencies are always respected in the SDVRP, we can generalize $\bar{c}$ to also include a penalty term for site violations:
\begin{equation*} 
\bar{c}^S\hvrpSol = \bar{c}\hvrpSol + \varphi^S s\hvrpSol = \bar{c}\hvrpSol + \varphi^S \sum_{h = 1}^H \max \left\{0, \sum_{i \in R_h} \mathbb{I}[\kappa_h \notin K_i] \right\},
\end{equation*}
where $\varphi^S$ is a large penalty coefficient. We note that the performance of the local search algorithm does not depend critically on the chosen values of the penalty coefficients $\varphi^Q$ and $\varphi^S$ as long as they are sufficiently large when compared to the transportation costs.

\section{Robust Integer Programming Model and Branch-and-Cut}\label{sec:milp_and_bnc}
The metaheuristic approaches described in the previous section determine high-quality robust feasible solutions, in general.
To precisely quantify the quality of these solutions however, we need a lower bound on the optimal objective value of the robust HVRP.
To that end, Section~\ref{sec:milp_and_bnc:ip_formulation} describes an integer programming (IP) formulation whose optimal solution coincides with that of the robust HVRP, while Section~\ref{sec:milp_and_bnc:bnc_algorithm} describes a branch-and-cut algorithm for its solution.

\subsection{Integer Programming Model}\label{sec:milp_and_bnc:ip_formulation}
Our model is similar to the classical \emph{vehicle-flow formulation} originally introduced in~\cite{Laporte1986:2vf} for the CVRP.
The model uses binary variables $y_{ik}$ to record if customer $i \in V_C$ is visited by a vehicle of type $k \in K$ and variables $x_{ijk}$ to record if the edge $(i, j) \in E$ is traversed by a vehicle of type $k \in K$.
To facilitate the description of the formulation, we define $V_k \coloneqq \left\{i \in V_C : \max_{q \in \mathcal{Q}} q_i \leq Q_k \right\}$ to be the subset of those customers that can be visited by a vehicle of type $k \in K$ under any customer demand realization $q \in \mathcal{Q}$.
For a given $S \subseteq V_k$, we also define $\delta_k(S)$ to be the subset of all edges in $E$ with one end point in $S$ and the other in $V_k \setminus S$.
The complete formulation is as follows.
\begingroup
\allowdisplaybreaks
\small
\begin{subequations}\label{eq:formulation}
\begin{alignat}{3}
&\displaystyle\mathop{\text{minimize}}_{x,y} &\;\;& \displaystyle\sum_{k \in K} \sum_{i \in V_C : (0, i) \in E} (f_k/2) x_{0ik} + \sum_{k \in K}\sum_{(i, j) \in E} c_{ijk} x_{ijk}  \label{eq:formulation:obj}\\
&\displaystyle\text{subject to} && \displaystyle y_{ik} \in \{0, 1\} &&\quad\forall i \in V_C,\; \forall k \in K, \label{eq:formulation:binary_y}\\
&&& \displaystyle x_{ijk} \in \{0, 1\} &&\quad\forall (i, j) \in E \cap (V_C \times V_C),\; \forall k \in K, \label{eq:formulation:binary_x_Vc}\\
&&& \displaystyle x_{0ik} \in \{0, 1, 2\} &&\quad\forall i \in V_C : (0, i) \in E,\; \forall k \in K, \label{eq:formulation:binary_x_0i}\\
&&& \displaystyle \sum_{k \in K} y_{ik} = \sum_{k \in K: i \in V_k} y_{ik} = 1 &&\quad\forall i \in V_C, \label{eq:formulation:exactly_one_vehicle_type}\\
&&& \displaystyle \sum_{j \in V : (i, j) \in E} x_{ijk} = 2y_{ik} &&\quad\forall i \in V_C,\; \forall k \in K, \label{eq:formulation:degrees}\\
&&& \displaystyle \sum_{i \in V_C : (0, i) \in E} x_{0ik} \leq 2m_k &&\quad \forall k \in K, \label{eq:formulation:fleet_size}\\
&&& \displaystyle \sum_{(i, j) \in \delta_k(S)} x_{ijk} + 2 \sum_{i \in S} (1 - y_{ik}) \geq 2\left\lceil{\frac{1}{Q_k} \max_{q \in \mathcal{Q}} \sum_{i \in S} q_i}\right\rceil &&\quad \forall S \subseteq V_k, \; \forall k \in K. \label{eq:formulation:rci}
\end{alignat}
\end{subequations}
\endgroup
The objective function~\eqref{eq:formulation:obj} minimizes the sum of the fixed costs and transportation costs; since the fixed costs are accounted on both the first and last edges of a route (i.e., both edges adjacent to the depot), the total must be divided by two to obtain the true fixed cost.
Constraints~\eqref{eq:formulation:binary_y}--\eqref{eq:formulation:binary_x_0i} enforce integrality restrictions.
Constraints~\eqref{eq:formulation:exactly_one_vehicle_type} stipulate that each customer must be visited by exactly one vehicle type; moreover, it also requires that this vehicle type be such that it can feasibly visit this customer under all possible demand realizations.
Constraints~\eqref{eq:formulation:degrees} enforce that if a customer is visited by vehicle type $k \in K$, then there are exactly two edges adjacent to it that are traversed by a vehicle of that type; moreover, if it is not visited by a vehicle of type $k \in K$, then all edge variables adjacent to it are set to zero.
Constraints~\eqref{eq:formulation:fleet_size} ensure that no more than $m_k$ vehicles of type $k \in K$ are used.
Constraints~\eqref{eq:formulation:rci} restrict \emph{subtours} (i.e., ensure that no vehicles of type $k \in K$ perform routes that are disconnected from the depot) and also enforce that the worst-case load of routes of type $k \in K$ are less than the capacity $Q_k$.
We refer to these inequalities as \emph{robust heterogeneous rounded capacity inequalities} since they generalize the classical \emph{rounded capacity inequalities} (RCI) for the deterministic CVRP~\cite{Lysgaard2004} (i.e., whenever $\mathcal{Q}$ and $K$ are both singletons).
Finally, we note that in the case of the SDVRP (see Table~\ref{table:summary_of_HVRP_variants}), the following additional constraint can be added to formulation~\eqref{eq:formulation} to ensure that all site dependencies are respected:
\begin{equation}\label{eq:formulation:sdvrp}\tag{\ref*{eq:formulation}i}
y_{ik} = 0 \quad \forall k \notin K_i,\; \forall i \in V_C.
\end{equation}
Proposition~\ref{prop:formulation_correctness} establishes the correctness of formulation~\eqref{eq:formulation}.
\begin{prop}\label{prop:formulation_correctness}
The feasible solutions of formulation~\eqref{eq:formulation} are in one-to-one correspondence with robust feasible solutions of the HVRP.
\end{prop}
\begin{proof}
Suppose $\hvrpSol$ is a robust feasible solution of the HVRP, i.e., it satisfies conditions~\ref{def:feasible:partition}, \ref{def:feasible:fleet_size} and~\ref{def:feasible:hvrp:robust_capacity}.
Construct the solution $(x, y)$ as follows: $y_{ik} = \sum_{h = 1}^H \mathbb{I}[i \in R_h] \mathbb{I}[k = \kappa_h]$ and $x_{ijk} = \sum_{h = 1}^H \sum_{l = 0}^{\abs{R_h}} \mathbb{I}[(i, j) = (r_{hl}, r_{hl + 1})] \mathbb{I}[k = \kappa_h]$.
We claim that $(x, y)$ is a feasible solution of formulation~\eqref{eq:formulation}.
To see this, first observe that satisfaction of constraints~\eqref{eq:formulation:binary_y}--\eqref{eq:formulation:degrees} follows from the definition of a route (see Section~\ref{sec:problem_definition}) and from the fact that $\hvrpSol$ satisfies condition~\ref{def:feasible:partition}.
Similarly, constraint~\eqref{eq:formulation:fleet_size} is satisfied because $\hvrpSol$ satisfies condition~\ref{def:feasible:fleet_size}.
Constraints~\eqref{eq:formulation:rci} are satisfied because of the following reason.
First, observe that we have:
\begingroup
\allowdisplaybreaks
\begin{align*}
\displaystyle \sum_{(i, j) \in \delta_k(S)} x_{ijk} &\geq 2 \Big\lvert\underbrace{h \in \{1, \ldots, H\} : \kappa_h = k \text{ and } S \cap R_h \neq \emptyset}_{\coloneqq H_k(S) = \text{index set of routes of type $k$ `crossing' $S$}}\Big\rvert \\
&= 2\left\lceil\frac{1}{Q_k} \sum_{h \in H_k(S)} Q_k \right\rceil \\
&\geq 2 \left\lceil \frac{1}{Q_k} \max_{q \in \mathcal{Q}} \sum_{h \in H_k(S)} \sum_{i \in S \cap R_h} q_i \right\rceil
= 
2 \left\lceil \frac{1}{Q_k} \max_{q \in \mathcal{Q}} \sum_{i \in S \cap \left(\cup_{h \in H_k(S)} R_h\right) } q_i \right\rceil,
\end{align*}
\endgroup
where the first inequality follows by construction of $x$ while the second inequality follows because \textit{(i)} $\hvrpSol$ satisfies condition~\ref{def:feasible:hvrp:robust_capacity} and \textit{(ii)} the maximum operator is subadditive.
Second, we have:
\begin{equation*}
2 \sum_{i \in S} (1 - y_{ik}) = 2 \abs{i \in S \setminus \left(\cup_{h \in H_k(S)} R_h\right)} \geq 2 \left\lceil \frac{1}{Q_k} \max_{q \in \mathcal{Q}} \sum_{i \in S \setminus \left(\cup_{h \in H_k(S)} R_h\right) } q_i \right\rceil,
\end{equation*}
where the equality follows by construction of $y$ while the inequality follows because \textit{(i)} each $i \in S \subseteq V_k$ satisfies $\max_{q \in \mathcal{Q}} q_i \leq Q_k$ and, \textit{(ii)} the maximum and ceiling operators are subadditive.
Finally, combining the above two expressions and using again the subadditivity of the maximum and ceiling operators shows that inequalities~\eqref{eq:formulation:rci} are satisfied.

Now, suppose that $(x, y)$ is a feasible solution of formulation~\eqref{eq:formulation}.
Construct $\hvrpSol$ as follows: \textit{(i)} $H \gets 0$; \textit{(ii)} for every $i \in V_C$, if $\sum_{k \in K} x_{0ik} = 1$ and $i \notin R_1, \ldots, R_H$, then set $H \gets H + 1$ and define $\kappa_{H} = \sum_{k \in K} k\mathbb{I}[y_{ik} = 1]$ and  $R_{H}$ to be the cycle that passes through customer $i$ in the subgraph of $G$ induced by $\left\{(i', j') \in E : x_{i'j'\kappa_{H}} = 1 \right\}$.
We claim that $\hvrpSol$ is a robust feasible solution of the HVRP.
To see this, first observe that $\hvrpSol$ satisfies condition~\ref{def:feasible:partition} because of inequalities~\eqref{eq:formulation:binary_y}--\eqref{eq:formulation:degrees}, and condition~\ref{def:feasible:fleet_size} because of inequality~\eqref{eq:formulation:fleet_size}.
To see that condition~\ref{def:feasible:hvrp:robust_capacity} is also satisfied for each route $R_h$, $h \in \{1, \ldots, H\}$: set $S = R_h$ and $k = \kappa_h$ in inequalities~\eqref{eq:formulation:rci}.
The left-hand side simplifies to $2$ while the right-hand side simplifies to $2\ceil{(1 / Q_{\kappa_h}) \max_{q \in \mathcal{Q}} \sum_{i \in S} q_i}$.
Hence, we have $1 \geq \ceil{(1 / Q_{\kappa_h}) \max_{q \in \mathcal{Q}} \sum_{i \in S} q_i}$, which implies that condition~\ref{def:feasible:hvrp:robust_capacity} is satisfied.
\end{proof}

\subsection{Branch-and-Cut Algorithm}\label{sec:milp_and_bnc:bnc_algorithm}
The number of variables in formulation~\eqref{eq:formulation} is $\mathcal{O}(mn^2)$ but the number of constraints is $\mathcal{O}(m2^n)$.
However, if we leave out the RCI constraints~\eqref{eq:formulation:rci}, then the number of remaining constraints is $\mathcal{O}(mn)$.
Therefore, we can solve the formulation in a \emph{cutting plane} fashion by removing constraints~\eqref{eq:formulation:rci} and dynamically re-introducing them if and when they are found to be violated by the solution of the current linear programming relaxation.
In fact, we can embed the cutting plane generation in each node of a branch-and-bound tree to obtain a \emph{branch-and-cut} algorithm.
We refer to~\cite{Lysgaard2004} for a general reference on branch-and-cut in the context of vehicle routing.
The performance of the branch-and-cut algorithm can be improved by adding in each tree node, inequalities that are valid but not necessary for the correctness of formulation~\eqref{eq:formulation}.
In the following, we describe several such valid inequalities as well as the associated \emph{separation algorithms}.
We also describe an effective preprocessing step that can reduce the number of variables in formulation~\eqref{eq:formulation}.

\paragraph{Valid inequalities.}
Several valid inequalities have been proposed for flow-based formulations of the deterministic HVRP in~\cite{Baldacci2009:FSMF,Yaman2006}.
Among these, the \emph{cover inequalities} and \emph{fleet-dependent capacity inequalities} are particularly effective in obtaining strong lower bounds.
The validity of the inequalities for the robust HVRP formulation~\eqref{eq:formulation} can be established by defining them for every possible customer demand realization $q \in \mathcal{Q}$.
To describe these inequalities, we assume, without loss of generality, that the vehicle types are sorted in increasing order with respect to their capacities: $Q_1 \leq \ldots \leq Q_m$.
We also define $s_k(q)$ to be total demand of the customers for which a vehicle of type $k$ is the smallest one that can visit them, under a particular customer demand realization $q \in \mathcal{Q}$.
Then, the following \emph{robust cover inequalities} are valid for formulation~\eqref{eq:formulation}.
\begin{equation}\label{eq:hvrp_cover}\tag{\ref*{eq:formulation}j}
\begin{array}{c}
\displaystyle\sum_{h = k}^m \left(\floor{\alpha Q_h} + \min\left\{1, \dfrac{\alpha Q_h -\left\lfloor{\alpha Q_h} \right\rfloor}{\alpha \sum\limits_{h = k}^m s_h(q) - \left\lfloor{\alpha \sum\limits_{h = k}^m s_h(q)} \right\rfloor } \right\} \right) \sum_{i \in V_C: (0, i) \in E}\hspace{-1em} x_{0ih} \geq 2 \left\lceil \alpha \sum_{h = k}^m s_h(q) \right\rceil \\
\displaystyle \hfill \forall \alpha \in \mathbb{R}_{+}, \; \forall k \in K, \; \forall q \in \mathcal{Q}.
\end{array}
\end{equation}
Let us define $\delta(S)$ to be the set of edges in $E$ with exactly one end point in $S$ and one end point in $V_C\setminus S$. Then, the following \emph{robust fleet-dependent capacity inequalities} are valid for formulation~\eqref{eq:formulation}.
\begin{equation}\label{eq:fleet_dependent}\tag{\ref*{eq:formulation}k}
\begin{array}{c}
\displaystyle \sum_{k \in K} \sum_{(i, j) \in \delta(S)} x_{ijk} + \sum_{h = k}^m \left\lceil 2 \left(\frac{Q_h - Q_{k - 1}}{Q_{k - 1}}\right) \right\rceil \sum_{i \in V_C: (0, i) \in E}\hspace{-1em} x_{0ih} \geq \left \lceil \frac{2}{Q_{k-1}} \max_{q \in \mathcal{Q}} \sum_{i \in S} q_i \right\rceil \\
\displaystyle \hfill \forall S \subseteq V_C, \; \forall k \in K \setminus \{1\}.
\end{array}
\end{equation}
The validity of the following \emph{generalized subtour elimination constraints} and \emph{generalized fractional capacity inequalities} can also be easily verified.
The term \textit{generalized} refers to the fact that these inequalities reduce to the classical subtour elimination constraints and fractional capacity inequalities, respectively, in the case of the deterministic CVRP (i.e., when both $\mathcal{Q}$ and $K$ are singletons).
However, unlike the latter, these inequalities do not dominate and are not dominated by the robust heterogeneous RCI constraints~\eqref{eq:formulation:rci}.
\begingroup
\allowdisplaybreaks
\begin{subequations}
\begin{gather}
\sum_{(i, j) \in \delta_k(S)} x_{ijk} \geq 2\max_{v \in S} y_{vk} \quad \forall S \subseteq V_k, \; \forall k \in K. \label{eq:hvrp_gsec}\tag{\ref*{eq:formulation}l} \\
\sum_{(i, j) \in \delta_k(S)} x_{ijk} \geq \frac{2}{Q_k} \sum_{i \in S} q_iy_{ik} \quad \forall S \subseteq V_k, \; \forall k \in K, \; \forall q \in \mathcal{Q}. \label{eq:hvrp_gfci}\tag{\ref*{eq:formulation}m}
\end{gather}
\end{subequations}
\endgroup
In addition to the above, any valid inequality for the \emph{two-index vehicle flow formulation} of the deterministic CVRP, defined over graph $G = (V, E)$ with vehicle capacity $Q = \max_{k \in K} Q_k$, can also be made valid for formulation~\eqref{eq:formulation} by defining it for all $q \in \mathcal{Q}$ and by replacing the two-index variable $x_{ij}$ with $\sum_{k \in K} x_{ijk}$ for all $(i, j) \in E$.
In our implementation, we used the comb, framed capacity and multistar inequalities in this manner~\cite{Lysgaard2004}.

\paragraph{Separation algorithms.}
Let $(\bar{x}, \bar{y})$ be a fractional solution encountered in some node of the search tree.
The goal of a separation algorithm is to identify if a particular member of a family of inequalities is violated by the current solution $(\bar{x}, \bar{y})$.
Consider the robust heterogeneous RCI constraints~\eqref{eq:formulation:rci}.
For a particular vehicle type $k \in K$, the identification of a customer set $S \subseteq V_k$ for which the corresponding inequality is violated by $(\bar{x}, \bar{y})$ is nontrivial.
In the deterministic CVRP, this is typically achieved by a local search algorithm. 
For example, a greedy search algorithm is presented in~\cite{Lysgaard2004}; this algorithm iteratively expands a (randomly initialized) set $S = \{s\}$ with a customer $j$ for which the corresponding slack of the RCI constraint (i.e., difference between the right-hand side and left-hand side) is maximized.
In our implementation, we extend this idea by using a tabu search procedure very similar to the one presented in Section~\ref{sec:local_search}.
The key difference is that a ``solution'' in the context of this local search algorithm is simply a customer set $S \subseteq V_k$ as opposed to an entire set of routes.
Specifically, the algorithm starts with a randomly selected customer set $S \subseteq V_k$ and then iteratively perturbs this set through a sequence of operations in which individual customers are added
or removed.
In each iteration, the algorithm greedily chooses a customer whose inclusion or removal maximizes the slack of the corresponding robust heterogeneous RCI constraint~\eqref{eq:formulation:rci}.
Similar to the argument in Section~\ref{sec:local_search}, computing this slack requires the computation of the right-hand side which, in turn, requires the efficient computation of the worst-case load over the current candidate set of customers $S$.
This is achieved by using the data structures described in Proposition~\ref{prop:cpu_and_memory_requirements}.
Finally, similar to Section~\ref{sec:local_search}, the algorithm also maintains tabu lists  of customers that have recently been added or removed to avoid cycling and to escape local optima. The algorithm terminates if we cannot maximize the slack of constraint~\eqref{eq:formulation:rci} for more than a pre-defined number of consecutive iterations.

The separation algorithm for the robust fleet-dependent capacity inequalities~\eqref{eq:fleet_dependent} is exactly the same as above.
The separation problem for the generalized subtour elimination constraints~\eqref{eq:hvrp_gsec} is solved using the polynomial-time algorithm described in~\cite{Fischetti1998}.
Similarly, the separation problem for the generalized fractional capacity inequalities~\eqref{eq:hvrp_gfci}, under a particular demand realization $q^\star \in \mathcal{Q}$, reduces to the separation problem of the fractional capacity inequalities for the deterministic CVRP if we define the customer demands to be $q^\star_i \bar{y}_{ik}$, which is known to be polynomial-time solvable~\cite{McCormick2003}.
In our implementation, we restrict the separation to a particular realization defined by $q^\star \in \arg\max_{q \in \mathcal{Q}} \sum_{i \in V_C} q^\star_i$.
Similarly, we separate the CVRP-based comb, framed capacity and multistar inequalities using the CVRPSEP package~\cite{Lysgaard2004} by only considering $q^\star \in \mathcal{Q}$.
Finally, the robust cover inequalities~\eqref{eq:hvrp_cover} are separated by enumerating $\alpha \in \{Q_1, \ldots, Q_m, \text{gcd}(Q_1, \ldots, Q_m)\}$, where $\text{gcd}$ denotes the greatest common divisor, and by considering only $q^\star \in \mathcal{Q}$.

\paragraph{Preprocessing.}
A simple method to reduce the number of vehicle types was presented in~\cite{Choi2007} for the deterministic HVRP. 
Suppose $UB$ is a known upper bound on the optimal objective value of formulation~\eqref{eq:formulation}.
For example, $UB$ may obtained using the metaheuristics described in Section~\ref{sec:local_search}.
Suppose we enforce now that at least one vehicle of type $k \in K$ must be used, by adding the constraint $\sum_{i \in V_C} y_{ik} \geq 1$ to formulation~\eqref{eq:formulation}.
If $LB'_k$ denotes a lower bound on the optimal value of this augmented problem and if $LB'_k > UB$, then we can delete vehicle type $k \in K$ and all of its associated variables from formulation~\eqref{eq:formulation}, since the corresponding solution would be suboptimal.
In our implementation, we estimate $LB'_k$ by solving the augmented formulation using a branch-and-cut algorithm and recording the global lower bound of the branch-and-bound tree after 1~minute.

\section{Computational Results}\label{sec:results}
This section presents computational results obtained using the metaheuristic algorithms described in Section~\ref{sec:local_search} as well as the exact algorithm described in Section~\ref{sec:milp_and_bnc}.
Specifically, Section~\ref{sec:results:instances} provides an overview of the benchmark instances used; Section~\ref{sec:results:heuristics} presents a detailed computational study using the ILS and AMP algorithms; Section~\ref{sec:results:lower_bounding} illustrates the quality of the lower bounds obtained using the branch-and-cut algorithm; and finally, in Section~\ref{sec:results:price_of_robustness}, we analyze the robust HVRP solutions in terms of their robustness and objective value.

All algorithms were coded in C++ and compiled using the GCC~7.3.0 compiler. Each run was conducted on a single thread of an Intel Xeon 3.1~GHz processor.
In our implementation of the ILS and AMP algorithms, the following parameter values were used: $\varphi^Q = 1000 \ubar{c} Q_\text{max}^{-1}$, $\varphi^S = 100 \ubar{c}$, $\chi = 10$, $\eta = 3$, $\nu = 30$, $\eta = 500$, $\delta = 0.5\ubar{c}$, $\theta = 0.7$ and $\mu = 16$, where we have defined $Q_\text{max} = \max_{k \in K} Q_k$ and $\ubar{c}_{ij} = \max_{(i, j) \in V_C \times V_C} \min_{k \in K} c_{ijk}$. An overall time limit of 1,000 seconds ($t_\text{lim} = 1000$) was used. In the implementation of the branch-and-cut algorithm, we used CPLEX~12.7 as the IP solver; all solver options were at their default values with three exceptions: \textit{(i)} general-purpose cutting planes and upper bounding heuristics were disabled, \textit{(ii)} strong branching was enabled, and \textit{(iii)} all valid inequalities described in Section~\ref{sec:milp_and_bnc:bnc_algorithm} were added using user-defined callback functions.
An overall time limit of 10,000~seconds was used in this case.

\subsection{Test Instances}\label{sec:results:instances}
All our instances are based on the following three benchmark datasets corresponding to different variants of the deterministic HVRP.
\begin{enumerate}[label={(\alph*)}]
\item HVRP instances: We consider the twelve instances involving up to 100 customers proposed in~\cite{Golden1984:fsm} and adapted by~\cite{Choi2007,Taillard1999:hvrp}.
The data of these instances can be found at \url{http://mistic.heig-vd.ch/taillard/problemes.dir/vrp.dir/vrp}.
The instances for the different HVRP variants are obtained by changing the data of the HVRPFD instances as follows, resulting in a total of 52 instances.
\begin{itemize}
\item HVRPD: Set $f_k = 0$ for each $k \in K$.
\item FSMFD: Set $m_k = n$ for each $k \in K$.
\item FSMD: Set $f_k = 0$ and $m_k = n$ for each $k \in K$.
\item FSMF: Set $m_k = n$ for each $k \in K$ and $c_{ijk} = e_{ij}$, where $e_{ij}$ is the Euclidean distance between nodes $i \in V$ and $j \in V$.
\end{itemize}

\item SDVRP instances: We consider the 13 instances containing up to 108 customers that have also been considered by~\cite{Nag1988,CordeauLaporte2001:sdvrp,Chao1999:sdvrp,Baldacci2009:MP}.
The data of these instances can be found at \url{http://neumann.hec.ca/chairedistributique/data/sdvrp}.

\item MDVRP instances: We consider the 9 instances involving up to 160 customers that have also been considered by~\cite{Cordeau1997:mdvrp,Baldacci2009:MP}.
The data of these instances can be found at \url{http://neumann.hec.ca/chairedistributique/data/mdvrp}.
\end{enumerate}

For each deterministic HVRP benchmark, we construct five classes of uncertainty sets. To ensure that the constructed sets are meaningful, we partition the customer set $V_C$ into four geographic quadrants, $NE$, $NW$, $SW$, and $SE$, based on the coordinates in the benchmark instance. Moreover, the customer demands specified in the benchmark are taken to be their nominal values $q^0$.
We then construct the following uncertainty sets, each of which is parametrized by scalars $\alpha, \beta \in [0, 1]$.
\begin{enumerate}[label={(\alph*)}]
\item Budget sets (originally proposed in~\cite{Gounaris2013:OR}):
\[
\mathcal{Q}_B = \left\{
q \in [(1 - \alpha)q^0, (1 + \alpha)q^0]: \sum_{i \in \Omega} q_i \leq (1 + \alpha \beta) \sum_{i \in \Omega} q_i^0 \;\;\; \forall \, \Omega \in \{NE,NW.SW,SE\}.
\right\}.
\]
This set stipulates that each customer demand can deviate by at most $\alpha \cdot 100\%$ from its nominal value, but the cumulative demand in each quadrant may not exceed its nominal value by more than $\beta \cdot 100\%$.

\item Factor models (originally proposed in~\cite{Gounaris2013:OR}):
\[
\mathcal{Q}_F = \left\{
q \in \mathbb{R}^n : q = q^0 + \Psi \xi \text{ for some } \xi \in \Xi_F
\right\}, \text{where }
\Xi_F = \left\{\xi \in [-1, 1]^4: \big\lvert{e^\top \xi}\big\rvert \leq 4\beta \right\}.
\]
This set models the demand of customer $i$ as a convex combination of 4~factors that can be interpreted as quadrant demands with the weights reflecting the relative proximity of customer $i$ to the quadrant. Specifically, we set $\Psi_{if} = \alpha q_i^0 \psi_{if} / \sum_{f' = 1}^4 \psi_{if}'$, where $\psi_{if}$ denotes the inverse distance between customer $i$ and the centroid of quadrant $f \in \{1, 2, 3, 4\}$.

\item Ellipsoids:
\[
\mathcal{Q}_E = \left\{
q \in \mathbb{R}^n : q = q^0 + \Sigma^{1/2} \xi \text{ for some } \xi \in \Xi_E
\right\}, \text{where }
\Xi_E = \left\{\xi \in \mathbb{R}^n: \xi^\top \xi \leq 1 \right\}.
\]
We define $\Sigma = (1 - \beta) \Psi \Psi^\top + \beta \text{diag}\left(\alpha q_1^0, \ldots, \alpha q_n^0 \right)^2$, where $\Psi$ is the factor loading matrix defined above while $\text{diag}(\cdot)$ is a square diagonal matrix with $(\cdot)$ denoting the entries along its main diagonal.
When $\beta = 0$, $\mathcal{Q}_E$ is approximated as a 4-dimensional ellipsoid centered at $q^0$ and the columns of $\Psi$ represent the directions along its semi-axes.
When $\beta \in (0, 1)$, $\mathcal{Q}_E$ is a general $n$-dimensional ellipsoid centered at $q^0$.
When $\beta = 1$, $\mathcal{Q}_E$ is an axis-parallel ellipsoid centered at $q_0$ with a semi-axis length of $\alpha q_i^0$ along the $i^\text{th}$ dimension; that is, when $\beta = 1$, $\mathcal{Q}_E$ inscribes the $n$-dimensional hyper-rectangle $[(1 - \alpha) q^0, (1 + \alpha) q^0]$.

\item Cardinality-constrained sets:
\[
\mathcal{Q}_G = \left\{
q \in [q^0, (1 + \alpha)q^0] : q = q^0 + \alpha(q^0 \circ \xi) \text{ for } \xi \in \Xi_G
\right\}, \text{where }
\Xi_G = \left\{\xi \in [0, 1]^n: e^\top \xi \leq \beta n \right\}.
\]
The set stipulates that each demand can deviate from its nominal value by up to $\alpha \cdot 100\%$ but the total number of customer demands that can simultaneously deviate is at most $\ceil{\beta n}$.

\item Discrete sets:
\[
\mathcal{Q}_D = \mathop{\text{conv}}\left( \left\{ q^0 \right\} \cup \left\{
q^{(j)} : j = 1,\ldots,\text{nint}(\beta n)
\right\}
\right).
\]
Here, $\text{nint}(\beta n)$ denotes the nearest integer to $\beta n$.
The points $q^{(j)}$ are generated by uniformly sampling $\text{nint}(\beta n)$ points from the $n$-dimensional hyper-rectangle $[(1 - \alpha) q^0, (1 + \alpha) q^0]$.
Thus, the set approximates the customer demands as independent, uniform random variables.
\end{enumerate}

The deterministic HVRP benchmarks are characterized by a high vehicle utilization under the nominal demands $q^0$; that is, the unused capacity of the vehicles is small, particularly in the case of problem variants with limited fleets (see Table~\ref{table:summary_of_HVRP_variants}).
If the fleet size and vehicle capacities are unchanged, then several benchmark instances become infeasible in the presence of demand uncertainty.
To alleviate this problem and conduct a meaningful computational study, we increase the capacity of each vehicle type $Q_k$ in each benchmark by 10\% (unless explicitly stated otherwise), which suffices to guarantee robust feasibility for $\alpha \leq 0.1$.

\subsection{Performance of Robust Local Search and Metaheuristics}\label{sec:results:heuristics}
The results reported in this section are averages across 10~runs for each of the 74 test instances.
For the budget sets, factor models and general ellipsoids, we set $(\alpha, \beta) = (0.1, 0.5)$, while for the cardinality-constrained and discrete sets, we set $(\alpha, \beta) = (0.1, 0.2)$.
We note that the axis-parallel ellipsoid is obtained by setting $(\alpha, \beta) = (0.1, 1)$.

Figure~\ref{figure:time_per_iteration} shows the time per local search iteration for the different classes of uncertainty sets described in Section~\ref{sec:wc_evaluation}.
We note here that each iteration of local search involves evaluating $\mathcal{O}(n^2)$ neighbor solutions, see Section~\ref{sec:local_search}.
We make the following observations from Figure~\ref{figure:time_per_iteration}.
First, the time per iteration correlates well with the time complexities described in Table~\ref{table:cpu_and_memory_requirements}.
Indeed, each local search iteration can be performed much faster when the uncertainty set is a budget set or axis-parallel ellipsoid since the worst-case load can be updated in constant time in such cases.
In contrast, the local search iterations are relatively slower when the uncertainty set is a discrete set or general ellipsoid for which the worst-case load can only be updated in linear time (linear in $D = \beta n$ and $n$ respectively).
Second, the results are remarkably similar across the ILS and AMP algorithms.
This shows that the time per iteration is dictated by the chosen uncertainty set and not by the overarching metaheuristic algorithm.

\begin{figure}[!htb]
\centering
\caption{Time per local search iteration under different classes of uncertainty sets (normalized with respect to the deterministic problem). The top and bottom graphs show results for the Iterated Local Search and Adaptive Memory Programming algorithms respectively.}
\label{figure:number_of_ls_iterations}
\includegraphics[scale=0.6]{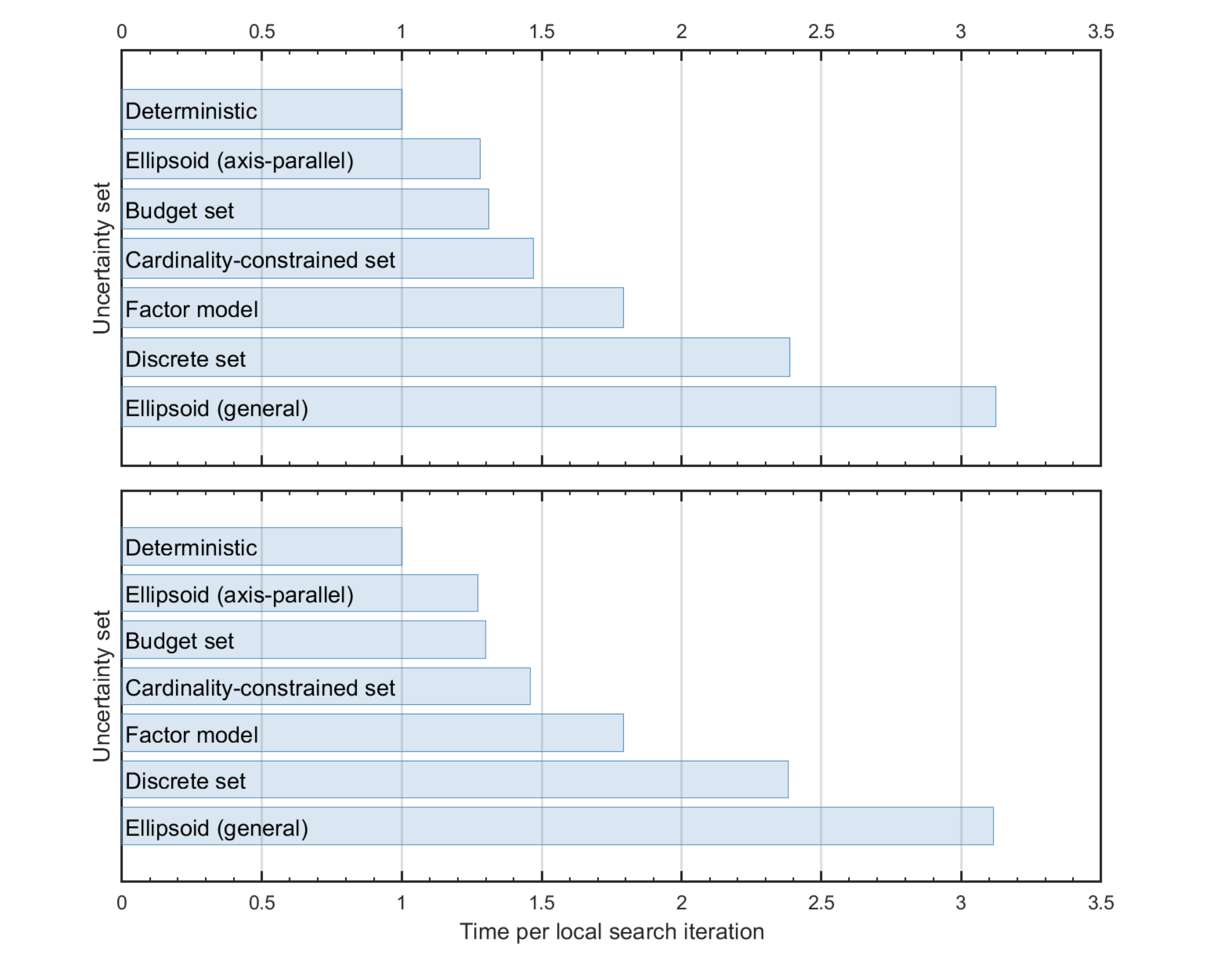}
\end{figure}

We note that the time limit of 1000 seconds is quite generous for all but the most difficult instances.
Indeed, in many cases, the metaheuristics have found good solutions in an early stage of the search process and have spent the remaining time trying to improve this solution. 
To see this, Figure~\ref{figure:progress_to_final} reports the percentage differences between the best solution $\hvrpSolx{B}$ (see Algorithms~\ref{algorithm:ILS} and~\ref{algorithm:AMP}) at various time points relative to the overall best solution (i.e., $\hvrpSolx{B}$ after 1000 seconds).
We make the following observations from Figure~\ref{figure:progress_to_final}.
First, for all uncertainty sets except the general ellipsoidal and discrete sets, the metaheuristics have found solutions that are within 1\% of the overall best after 10~seconds, and there is practically no improvement in the best solution after five minutes.
For the general ellipsoid and discrete sets, the solutions are within 2\% of the overall best after 10~seconds.
This indicates that it is probably better to restart the algorithm with a different random seed, rather than improve the solution, at this point.
Second, and similar to Figure~\ref{figure:time_per_iteration}, the performance across the various uncertainty sets correlates well with the complexities reported in Table~\ref{table:cpu_and_memory_requirements}, while the performance across the two metaheuristics is similar.

\begin{figure}[!htb]
\centering
\caption{Average progress of the metaheuristic solutions under different classes of the uncertainty sets. The top and bottom graphs show results for the Iterated Local Search and Adaptive Memory Programming algorithms respectively. The graphs report the average percentage differences relative to the overall best solution after 1000~seconds.}
\label{figure:progress_to_final}
\includegraphics[scale=0.6]{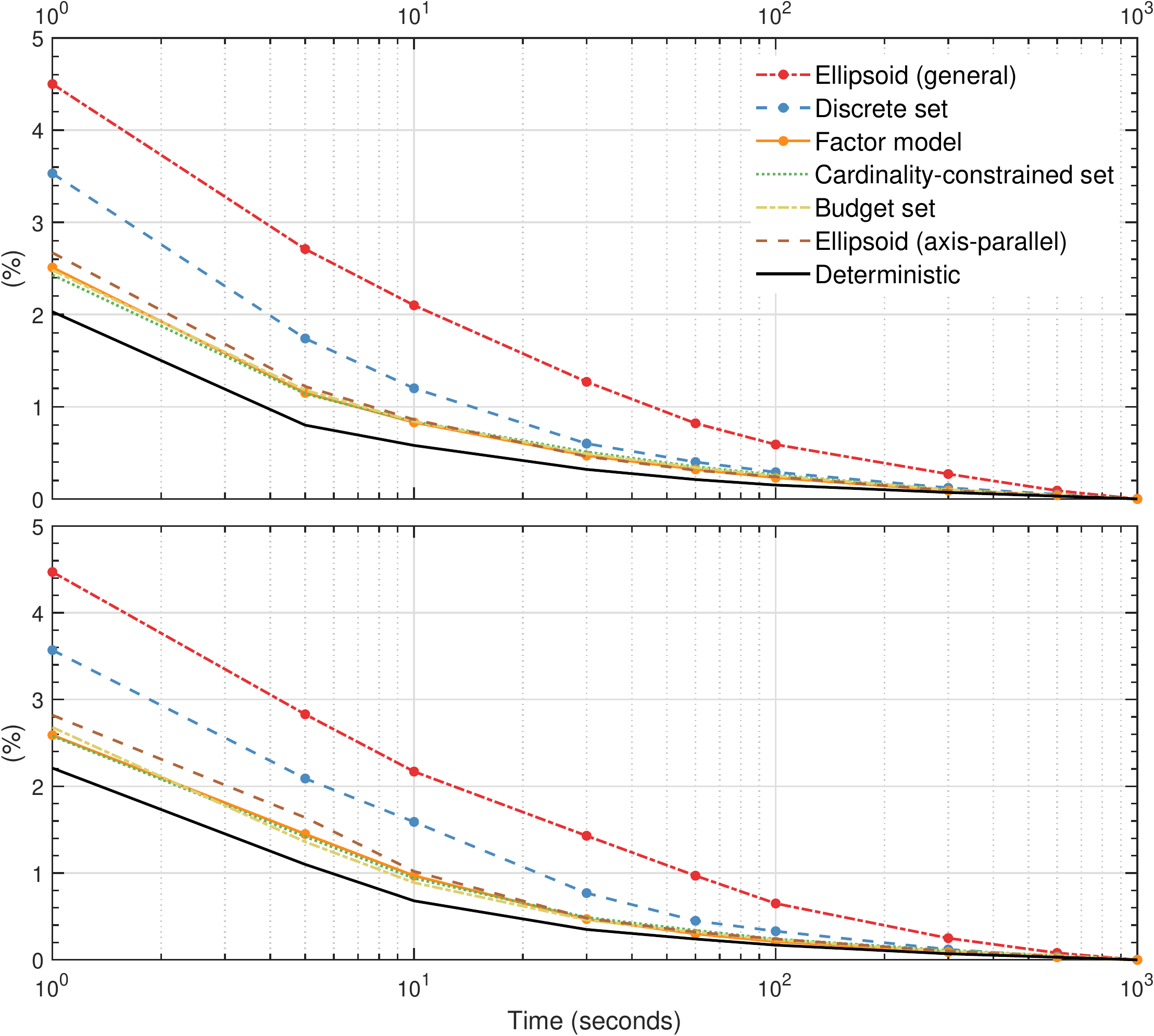}
\end{figure}

Finally, to understand the sensitivity of the proposed metaheuristic algorithms to the initial random seed, Figure~\ref{figure:random_seed_effect} plots the average percentage differences of each run relative to the overall best solution of the 10 runs, across all the 74 test instances.
The figure shows that both the ILS and AMP algorithms are fairly robust across all classes of uncertainty sets.
Indeed, the median deviation is less than 0.2\% across all uncertainty sets and in several cases, it is less than 0.1\%.

\begin{figure}[!htb]
\centering
\caption{Average deviation of the metaheuristic solutions (with respect to the best solution of 10~runs) under different classes of uncertainty sets. The top and bottom graphs show results for the Iterated Local Search and Adaptive Memory Programming algorithms respectively. The red mark corresponds to the median, the upper edge of the box to the 75$^\text{th}$ percentile, while the uppermost mark to the maximum of (74$\times$10) runs not considered outliers (which are indicated by green dots).}
\label{figure:random_seed_effect}
\includegraphics[scale=0.6]{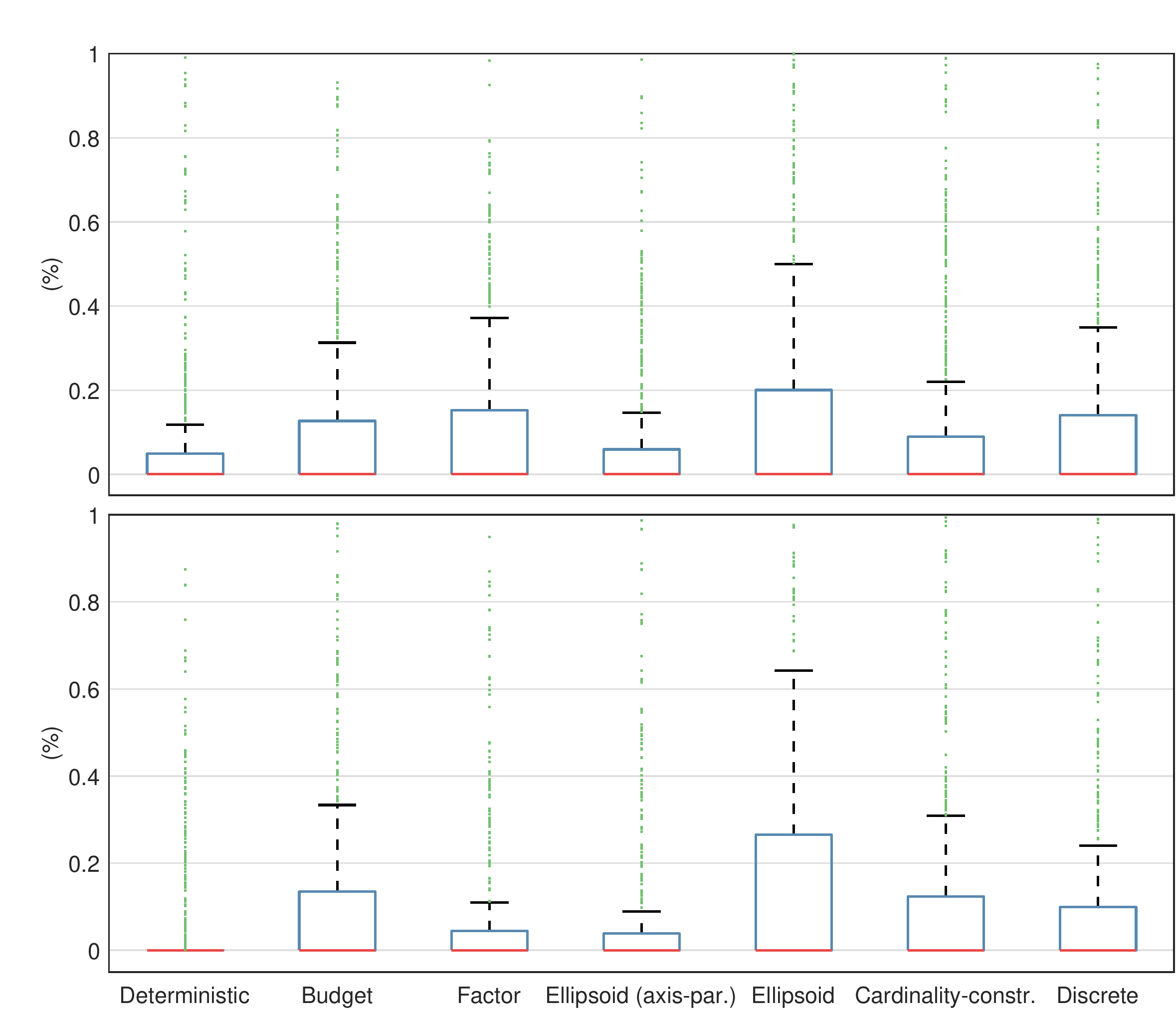}
\end{figure}

\subsection{Quality of Lower Bounds}\label{sec:results:lower_bounding}
Table~\ref{table:guaranteed_optimality_gaps} reports the quality of the lower bounds obtained using the IP formulation~\eqref{eq:formulation} described in Section~\ref{sec:milp_and_bnc}.
Specifically, the entries report the guaranteed optimality gaps, which is defined for a given instance as $(z_{ub} - z_{lb}) / z_{ub} \times 100\%$, where $z_{ub}$ is objective value of the best solution found across all 10$\times$2 runs of the ILS and AMP metaheuristics after 1000~seconds, while $z_{lb}$ is the global lower bound of the branch-and-cut algorithm after 10,000~seconds.
The table shows that the lower bounds for a given uncertainty set $\mathcal{Q}$ are very close to what one can expect for the deterministic problem (indicated by $\left\{q^0\right\}$).
Indeed, the average optimality gap for the deterministic problem is 5.9\%, while the average optimality gap for the robust problem varies between $6.7\%$ and $8.9\%$, on average.
Moreover, the lower bounding is particularly effective for the FSMFD, FSMF and MDVRP variants, i.e., whenever the fleet size is unlimited and either fixed costs or homogeneous fleets are considered.
In such cases, the average optimality gap is less than 5\% across all uncertainty sets.

% Table generated by Excel2LaTeX from sheet 'exact_summary'
\begin{table}[!htb]
  \centering
  \small
  \caption{Guaranteed optimality gaps of the metaheuristic solutions (in percent) across all benchmarks of the HVRP variants from Table~\ref{table:summary_of_HVRP_variants} and across different classes of uncertainty sets.}
    \begin{tabularx}{\textwidth}{lRRRRRRR}
    \toprule
           & $\left\{q^0\right\}$ & $\mathcal{Q}_B$ & $\mathcal{Q}_F$ & $\mathcal{Q}_E^\text{ax}$ & $\mathcal{Q}_E^\text{gen}$ & $\mathcal{Q}_G$ & $\mathcal{Q}_S$ \\
    \midrule
    HVRPFD & 7.25   & 9.91   & 9.03   & 9.03   & 8.10   & 11.89  & 9.10 \\
    HVRPD  & 11.56  & 13.42  & 13.23  & 13.12  & 12.59  & 14.73  & 13.12 \\
    FSMFD  & 3.98   & 6.43   & 5.22   & 6.38   & 4.95   & 7.49   & 6.64 \\
    FSMD   & 8.53   & 9.68   & 9.46   & 9.51   & 9.34   & 10.48  & 9.57 \\
    FSMF   & 2.98   & 5.65   & 4.45   & 4.96   & 3.72   & 6.68   & 5.36 \\
    SDVRP  & 5.30   & 7.06   & 6.74   & 5.90   & 5.84   & 7.60   & 6.19 \\
    MDVRP  & 3.58   & 5.24   & 4.86   & 4.71   & 4.61   & 5.80   & 4.85 \\ \midrule
    All    & 5.91   & 7.93   & 7.28   & 7.38   & 6.74   & 8.92   & 7.58 \\
    \bottomrule
    \end{tabularx}%
  \label{table:guaranteed_optimality_gaps}%
\end{table}%

We note that the reported entries in Table~\ref{table:guaranteed_optimality_gaps} are very conservative and they are limited by the lower bounds from the branch-and-cut algorithm rather than the upper bounds from the metaheuristics.
In fact, we believe that the metaheuristic solutions are near-optimal.
There are two reasons for this.
First, the branch-and-cut algorithm was never able to find a solution that was better than the provided metaheuristic solutions.
Second, we also used our metaheuristics to obtain solutions for all of the 74 original, deterministic benchmarks by setting the vehicle capacities $Q_k$ to their original values.
Table~\ref{table:original_deterministic_benchmarks} reports the aggregated results.
For each problem variant, the column \textbf{\#} reports the number of instances, \textbf{Gap (\%)} reports the average percentage difference between the obtained solution and the \textit{best known solution} (BKS) taken from~\cite{Pessoa2018:hvrp}, \textbf{BKS found} reports the number of instances for which the obtained solution matched the best known solution and \textbf{Time (sec)} reports the average time to find the obtained solution.
The table shows that even under the deterministic setting, both metaheuristics are very competitive when compared to existing algorithms for the HVRP (e.g., see~\cite{Penna2017}).
The AMP algorithm is superior to the ILS as it is able to match 64 out of 74 best known solutions with an average gap of 0.1\%.

% Table generated by Excel2LaTeX from sheet 'original_deterministic_summary'
\begin{table}[!htb]
  \centering
  \small
  \caption{Summary of results obtained using the metaheuristic algorithms on the 74 original benchmark instances of the deterministic HVRP.}
    \begin{tabularx}{\textwidth}{lr*{8}{R}}
    \toprule
           &        & \multicolumn{4}{c}{ILS}           & \multicolumn{4}{c}{AMP} \\
           \cmidrule(r){3-6}\cmidrule(l){7-10}
           &        & \multicolumn{2}{c}{Best run} & \multicolumn{2}{c}{Average} & \multicolumn{2}{c}{Best run} & \multicolumn{2}{c}{Average} \\
           \cmidrule(r){3-4}\cmidrule(lr){5-6}\cmidrule(lr){7-8}\cmidrule(l){9-10}
           & \#     & Gap (\%) & BKS found & Gap (\%) & Time (sec) & Gap (\%) & BKS found & Gap (\%) & Time (sec) \\
           \midrule
    HVRPFD & 8      & 0.19   & 5      & 0.31   & 325    & 0.19   & 6      & 0.30   & 283 \\
    HVRPD  & 8      & 0.21   & 6      & 0.38   & 271    & 0.00   & 8      & 0.20   & 184 \\
    FSMFD  & 12     & 0.10   & 8      & 0.18   & 244    & 0.02   & 9      & 0.12   & 192 \\
    FSMD   & 12     & 0.10   & 10     & 0.19   & 113    & 0.00   & 11     & 0.03   & 134 \\
    FSMF   & 12     & 0.09   & 9      & 0.20   & 221    & 0.02   & 10     & 0.10   & 242 \\
    SDVRP  & 13     & 0.07   & 9      & 0.22   & 274    & 0.00   & 12     & 0.17   & 189 \\
    MDVRP  & 9      & 0.10   & 4      & 0.23   & 228    & 0.01   & 8      & 0.17   & 183 \\ \midrule
    All    & 74     & 0.12   & 51     & 0.23   & 234    & 0.03   & 64     & 0.14   & 198 \\
    \bottomrule
    \end{tabularx}%
  \label{table:original_deterministic_benchmarks}%
\end{table}%

\subsection{Price of Robustness for Different Classes of Uncertainty Sets}\label{sec:results:price_of_robustness}
In this section, we quantify the average increase in total cost of the robust HVRP solution compared to its deterministic counterpart.
For each of the five classes of uncertainty sets, we fix $\alpha = 0.1$ and vary $\beta$ between $0$ and $1$.
For each case, we estimate the cost of the robust solution for each of the 74 test instances using a single run of the AMP metaheuristic under a time limit of 1000 seconds.
We then compare the total cost of the obtained solution with that of the deterministic instance (obtained by setting $(\alpha, \beta) = (0, 0)$).
Figure~\ref{figure:price_of_robustness} reports the results of this sensitivity analysis.

\begin{figure}[!htb]
\centering
\caption{Average percentage increase in transportation costs relative to the deterministic problem, over different classes of uncertainty sets with $\alpha = 0.1$. The marked square represents the cost increase when the uncertainty set is the $n$-dimensional hyperrectangle $[(1-\alpha)q^0, (1 + \alpha)q^0]$ while the marked ellipse represents the cost increase when the uncertainty set is the axis-parallel ellipsoid that inscribes this $n$-dimensional hyperrectangle.}
\label{figure:price_of_robustness}
\includegraphics[scale=0.62]{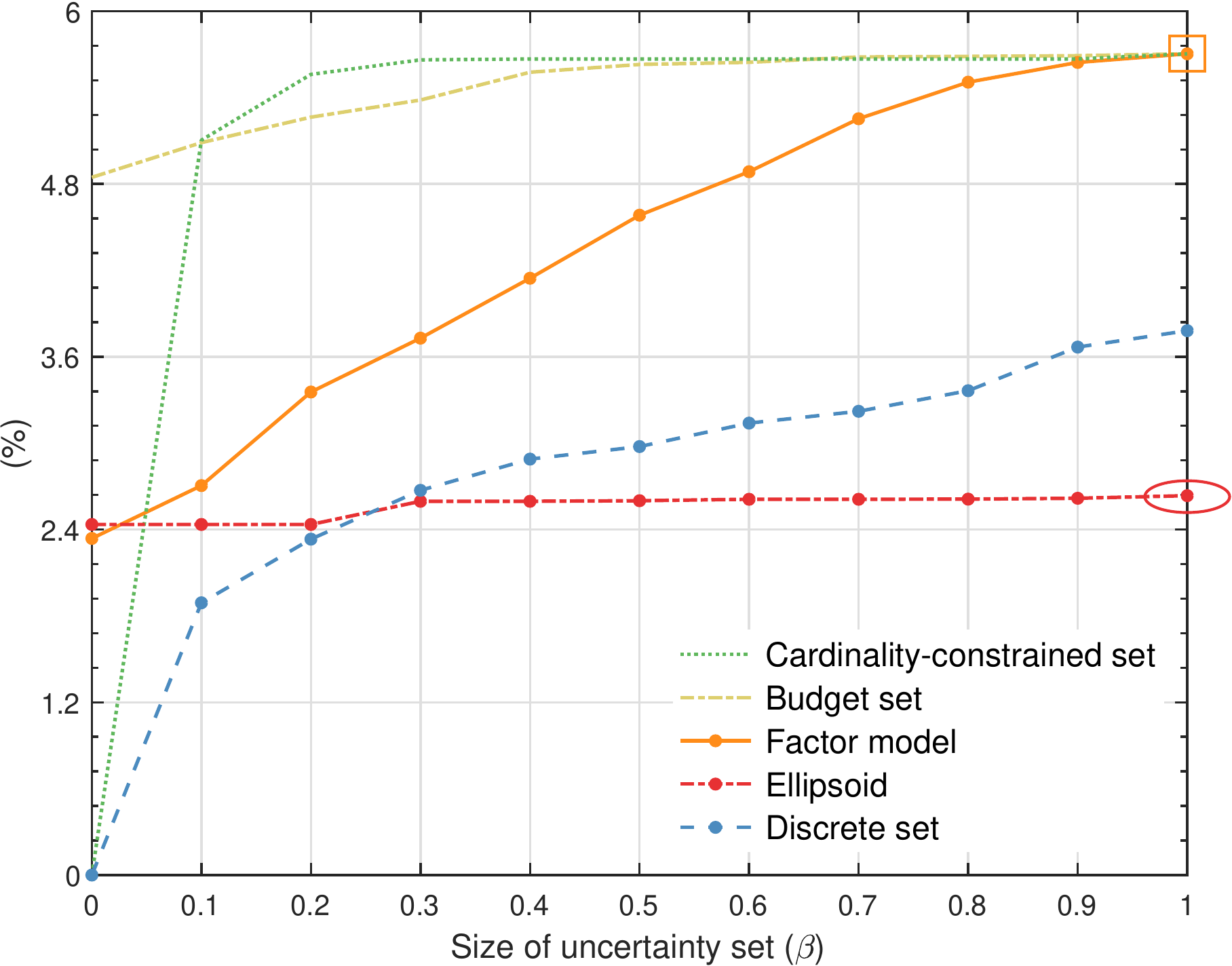}
\end{figure}

Figure~\ref{figure:price_of_robustness} offers several interesting insights.
First, the robust solutions are only slightly more expensive than their deterministic counterparts.
Even when the uncertainty set is a hyperrectangle where every customer can attain their worst realization independently, one can obtain modestly expensive solutions ($\approx6\%$) at the benefit of being immunized against a considerable random increase in customer demands (up to 10\%).
This cost increase can be reduced to $\approx2\%$ by controlling the size and shape of the uncertainty set.
Second, taking into account the results in Figure~\ref{figure:time_per_iteration}, it appears that different uncertainty sets offer different levels of trade-off between cost, robustness and tractability.
Indeed, while discrete sets offer the greatest flexibility in costs, they come at a considerable increase in computational complexity.
In contrast, the factor model also appears to offer significant flexibility but at a far less increase in complexity.
Finally, the added computational burden in modeling a non-axis-parallel ellipsoid does not pay off in terms of flexibility in transportation costs; indeed, the axis-parallel ellipsoids are only marginally more expensive at the benefit of being extremely tractable in the context of local search.

\section{Concluding Remarks}\label{sec:conclusions}
In this paper, we studied a broad class of heterogeneous fleet vehicle routing problems under the setting where the customer demands are not known precisely when the fleet composition and routes must be decided. 
We modeled the unknown customer demands as random variables that can take values in any of five broad classes of practically-relevant uncertainty sets. To hedge against this uncertainty, we aimed to determine a solution that is robust, i.e., a solution that remains feasible for all anticipated demand realizations, and we elucidated that the efficiency in computing such robust solutions (both in a heuristic and exact manner) lies in the ability to efficiently compute the worst-case loads of vehicle routes over the given uncertainty set.
With this insight, we capitalized on well-known local search algorithms for deterministic vehicle routing and augmented them with appropriate data structures to generate robust solutions. We then illustrated how to incorporate the proposed local search in two metaheuristic implementations.
Finally, we quantified the quality of the metaheuristic solutions using lower bounds obtained from an integer programming formulation.

Our study offers several practical insights.
First, robust solutions can be obtained with similar computational effort as deterministic solutions, and this is especially true in the context of metaheuristics.
In fact, since robust solutions can be obtained by augmenting local search in a modular fashion, the deployment of these advances in existing codes becomes straightforward.
Second, the trade-off between robustness and cost is highly dependent on the choice of the uncertainty set.
While some uncertainty sets might offer greater modeling flexibility or the ability to make better use of available data, they might not necessarily allow for a smooth variation in transportation costs as a function of their size. Moreover, the computational tractability of incorporating them in the solution algorithm (whether exact or heuristic) must also be carefully assessed before making a final choice.

\section*{Acknowledgements}
The authors gratefully acknowledge support from the U.S. National Science Foundation (award numbers CMMI-1434432 and CMMI-1434682). Panagiotis Repoussis also acknowledges support from the Athens University of Economics and Business (award numbers EP-2536-01 and EP-2855-01). Finally, Anirudh Subramanyam acknowledges support from the John and Claire Bertucci Graduate Fellowship Program.

\bibliographystyle{plain}
\bibliography{bibliography}

\newpage

\begin{appendices}

%\section{Metaheuristic Algorithms}

\section{Iterated Local Search}\label{sec:local_search:ILS}
Iterated Local Search (ILS), as the name implies, refers to the repeated application of local search to a current solution.
The current solution may either be generated from scratch using a \emph{construction heuristic} or by \emph{perturbing} a locally optimal solution.
We refer the reader to~\cite{Lourenco2003:ILS} for an introduction to this subject.
Our specific ILS implementation for the robust HVRP is described in Algorithm~\ref{algorithm:ILS}.

\begin{algorithm}[htb]
\caption{Iterated local search.}
\label{algorithm:ILS}
\begin{algorithmic}[1]
\Require $\chi$, $\eta$, $\nu$, $\zeta$, and $\delta$ (user-defined parameters) 
\State Start timer $t$, $\hvrpSolx{B} \gets (\emptyset, \emptyset)$
\While{$t < t_\text{lim}$}
% Construction phase
\State $\hvrpSol \gets $ \Call{Construct Solution}{$\eta$}\Comment{Construction phase}
\State $\hvrpSol \gets $ \Call{Tabu Search}{$\hvrpSol, \nu, \zeta$}
\IIf{$\bar{c}\hvrpSol < \bar{c}\hvrpSolx{B}$} $\hvrpSolx{B} \gets \hvrpSol$ \EndIIf
% Perturbation phase
\State $\text{counter} \gets 0$, $\hvrpSolx{\prime} \gets \hvrpSol$\Comment{Perturbation phase}
\While{$\text{counter} < \chi$}
\State $\hvrpSol \gets $ \Call{Perturb Solution}{$\hvrpSol, \delta$}
\State $\hvrpSol \gets $ \Call{Tabu Search}{$\hvrpSol, \nu, \zeta$}
\IIf{$\bar{c}\hvrpSol < \bar{c}\hvrpSolx{B}$}
$\hvrpSolx{B} \gets \hvrpSol$
\EndIIf
\If{$\bar{c}\hvrpSol < \bar{c}\hvrpSolx{\prime}$}
\State $\text{counter} \gets 0$
\State $\hvrpSolx{\prime} \gets \hvrpSol$
\Else \State $\text{counter} \gets \text{counter} + 1$
\EndIf
\EndWhile
\EndWhile
\State \Return $\hvrpSolx{B}$
\end{algorithmic}
\end{algorithm}

Algorithm~\ref{algorithm:ILS} consists of two phases: a \emph{construction} phase (lines~3--5) and a \textit{perturbation} phase (lines~6--17).
The construction phase first constructs an initial solution (line~3) and then improves it using an efficient local search algorithm called \emph{tabu search} (line~4).
The perturbation phase iteratively perturbs this solution (line~8) and improves it using tabu search (line~9).
The perturbation phase terminates if it fails to encounter a solution that is better than the best one found in the current iteration $\hvrpSolx{\prime}$ for more than $\chi$ attempts (line~7).
The ILS algorithm terminates after a pre-specified time limit $t_\text{lim}$ is reached (line~2), at which point the best encountered solution $\hvrpSolx{B}$ is returned (line~19).
The parameters $\eta$, $\nu$, $\zeta$ and $\delta$ are used as inputs to the construction heuristic, tabu search and perturbation mechanisms, respectively, which we describe next.

\paragraph{Construction heuristic.}
The procedure \textsc{Construct Solution}$(\eta)$ works by gradually inserting customers into an initially empty solution.
At any given iteration, an empty route is first constructed for each vehicle type $k \in K$.
To ensure that the fleet availability constraint~\ref{def:feasible:fleet_size} is satisfied, this is done only if the number of routes of vehicle type~$k$ in the current partial solution is less than its available number $m_k$.
Assuming this is the case, we keep adding unrouted customers to this route until it is no longer possible to do so (i.e., either because all customers have been routed or because the capacity condition~\ref{def:feasible:hvrp:robust_capacity} would be violated).
Specifically, all customers that can be potentially added to the route are first inserted into a \emph{restricted candidate list}; a random customer is then selected from this list and inserted into the position that greedily minimizes the corresponding \emph{insertion cost}.
In our implementation, the restricted candidate list is cardinality-based and fixed to a pre-defined size $\eta$.
The parameter $\eta$ determines the extent of randomization and greediness during the construction process.
Based on empirical evidence, a value of $\eta \in [1, 10]$ appeared to work well in our numerical experiments. 

If a route was successfully constructed for at least one vehicle type, we select the route $R$ of vehicle type $k$ for which the \emph{average cost per unit of carried load}, defined as $c(R, k)/\max_{q \in \mathcal{Q}} \sum_{i \in R} q_i$, is minimized.
If no route could be constructed, then any remaining unrouted customer is inserted into an existing route (and corresponding position) for which a randomly weighted sum of the capacity violation and insertion cost is minimized.
We remark that efficient computation of the average cost per unit of carried load is enabled by the data structures described in Proposition~\ref{prop:cpu_and_memory_requirements}.

\paragraph{Tabu search.}
In principle, we can replace all calls to the \textsc{Tabu Search}$(\hvrpSol, \nu, \zeta)$ procedure by a simple local search algorithm that iteratively explores each pre-defined neighborhood to determine an improving solution.
However, doing so will result in solutions that are only locally optimal with respect to the given neighborhoods, i.e., all neighbor solutions $\hvrpSolx{\prime} \in \Omega_Y\hvrpSol$ in each of the pre-defined neighborhoods $Y$ are non-improving: $\bar{c}\hvrpSolx{\prime} \geq c\hvrpSol$.
Tabu search~\cite{Glover1997} overcomes this shortcoming of local search and enhances its performance in two important ways: \textit{(i)} non-improving moves are allowed, and \textit{(ii)} improving moves may be disallowed.
This enhancement is achieved using a short term memory (also known as a \emph{tabu list}) to keep track of the most recently visited solutions in the search history and to prevent revisiting them for a predefined number of local search iterations $\nu$ (the \textit{tabu tenure}).
Any potential solution that has been visited within the last $\nu$ iterations is marked ``tabu'' (or forbidden) and inserted into the tabu list, so that the algorithm does not cycle by repeatedly visiting the same solutions.
In fact, an admissible neighbor solution that is in the tabu list can be visited only if certain \emph{aspiration criteria} are met; specifically, the tabu status of a solution is overridden only if it improves upon the best encountered solution.
Moreover, the tabu search terminates if it performs $\zeta$ local search iterations without observing any further improvement.
Typical values for these parameters are $\nu \in [20, 40]$ and $\zeta \in [100, 500]$ and in our implementation, we set $\nu = 30$ and $\zeta = 500$.
Furthermore, we considered the set of neighborhoods to be the (intra- and inter-route) 1-0 relocate, 1-1 exchange and 2-opt neighborhoods.
At each iteration, we randomly selected one of these neighborhoods and traversed it in lexicographic order, applying pruning mechanisms based on both feasibility and gain.
The first improving neighbor solution replaced the current solution.
Since the size of each of these neighborhoods is $\mathcal{O}(n^2)$, each iteration of tabu search itself can take $\mathcal{O}(n^2)$ time in the worst case, and its run time is largely dictated by the run time of the local search algorithm described in the previous section.

\paragraph{Perturbation mechanism.}
The procedure \textsc{Perturb Solution}$(\hvrpSol, \delta)$ attempts to perturb the current solution $\hvrpSol$ such that the new solution cannot be encountered by application of tabu search alone.
In our implementation, we consider a perturbation mechanism that first removes the route $R$ of vehicle type $k$ from the current solution which has the maximum value of \emph{average cost per unit of carried load}.
In addition to this route, the mechanism also removes any routes $R'$ that are ``sufficiently close'' to $R$, i.e., routes $R'$ for which $\text{distance}(R, R') \coloneqq \max_{(i, j) \in R \times R'} \ubar{c}_{ij} < \delta$.
Here, $\ubar{c}_{ij}$ is any suitably defined distance measure between customers $i$ and $j$ (e.g., geographical distance between $i$ and $j$), and in our implementation, we set $\ubar{c}_{ij} = \min_{k \in K} c_{ijk}$.
If all routes $R'$ satisfy $\text{distance}(R, R') \geq \delta$, then the route $R''$ which has the smallest distance to $R$ is removed from the current solution.
All customers that were visited on the deleted routes are then considered to be unrouted and are added back to the current partial solution using the same greedy insertion procedure that was used in the construction heuristic.
We note that the parameter $\delta$ determines the extent of perturbation, with larger values of $\delta$ corresponding to higher extents of perturbation.

\section{Adaptive Memory Programming}\label{sec:local_search:AMP}
Adaptive Memory Programming (AMP) is a metaheuristic that focuses on the exploitation of strategic memory components.
Based on the intuition that high-quality locally optimal solutions share common features (e.g., common customer visiting sequences), AMP attempts to exploit a set of long-term memories (in contrast to the short-term memory used in tabu search) for the iterative construction of new \emph{provisional solutions}. These solutions are used to restart and intensify the search, while adaptive learning mechanisms are
applied to update the memory structures. We refer the reader to~\cite{Glover1997,Taillard2001} for a general overview of this subject.
Our specific AMP implementation for the robust HVRP is described in Algorithm~\ref{algorithm:AMP}.

\begin{algorithm}[htb]
\caption{Adaptive Memory Programming.}
\label{algorithm:AMP}
\begin{algorithmic}[1]
\Require $\mu$, $\eta$, $\nu$, $\zeta$, and $\theta$ (user-defined parameters) 
\State Start timer $t$, $\mathcal{P} \gets \emptyset$, $\hvrpSolx{B} \gets (\emptyset, \emptyset)$
% Initialization phase
\While{$\abs{\mathcal{P}} < \mu$}\Comment{Initialization phase}
\State $\hvrpSol \gets $ \Call{Construct Solution}{$\eta$}
\State $\hvrpSol \gets $ \Call{Tabu Search}{$\hvrpSol, \nu, \zeta$}
\IIf{$\bar{c}\hvrpSol < \bar{c}\hvrpSolx{B}$} $\hvrpSolx{B} \gets \hvrpSol$ \EndIIf
\State $\mathcal{P} \gets \mathcal{P} \cup \hvrpSol$ 
\EndWhile
% Exploitation phase
\While{$t < t_\text{lim}$}\Comment{Exploitation phase}
\State $\hvrpSol \gets $ \Call{Construct Provisional Solution}{$\mathcal{P}, \theta$}
\State $\hvrpSol \gets $ \Call{Tabu Search}{$\hvrpSol, \nu, \zeta$}
\IIf{$\bar{c}\hvrpSol < \bar{c}\hvrpSolx{B}$} $\hvrpSolx{B} \gets \hvrpSol$ \EndIIf
\State $\mathcal{P} \gets $ \Call{Update Reference Set}{$\mathcal{P}, \hvrpSol$}
\EndWhile
\State \Return $\hvrpSolx{B}$
\end{algorithmic}
\end{algorithm}

Algorithm~\ref{algorithm:AMP} consists of two phases: an \emph{initialization} phase (lines~2--7) and an \emph{exploitation} phase (lines~8--13).
The initialization phase populates the \emph{reference set} $\mathcal{P}$ with $\mu$ solutions that are generated by first constructing an initial solution (line~3) and then improving it using tabu search (line~4).
Once the initialization phase has completed, the exploitation phase manipulates $\mathcal{P}$ by exploring search trajectories initiated using the provisional solutions as starting points.
Specifically, at each iteration of the exploitation phase, a provisional solution is first constructed by identifying common features of the reference solutions in $\mathcal{P}$ (line~9).
This provisional solution is then further improved using tabu search (line~10) and inserted into the reference set $\mathcal{P}$ (line~12).
The AMP algorithm terminates after a pre-specified time limit $t_\text{lim}$ is reached (line~8), at which point the best encountered solution $\hvrpSolx{B}$ is returned (line~14).
The procedures \textsc{Construct Solution}$(\eta)$ and \textsc{Tabu Search}$(\hvrpSol, \nu, \zeta)$ along with their associated parameters $\eta$, $\nu$ and $\zeta$ are exactly the same as in ILS (see Algorithm~\ref{algorithm:ILS}).
The parameters $\theta$ and $\mu$ are used in the provisional construction and reference set update methods respectively, which we describe next.

\paragraph{Generation of provisional solutions.}
Provisional solutions are constructed by identifying and combining \emph{elite components} from the reference set $\mathcal{P}$. We define an elite component to be a route associated with a particular vehicle type whose edges appear ``sufficiently frequently'' among the solutions in $\mathcal{P}$.
The overall procedure \textsc{Construct Provisional Solution}$(\mathcal{P}, \theta)$ works as follows.
We first attempt to generate a route for every vehicle type $k \in K$, assuming the number of routes of this type is less than $m_k$ in the current solution.
With probability $\theta$, we construct a route using the members of $\mathcal{P}$, and with probability $1 - \theta$, we use the mechanism outlined in the basic \textsc{Construct Solution}$(\eta)$ procedure.
In the former case, we first assign a score to each route $R_h = (r_{h1}, \ldots, r_{h\abs{R}})$ of vehicle type $\kappa_h$ from the solution $\hvrpSol \in \mathcal{P}$ as follows: $\text{score}(R_h, \kappa_h) = \text{w}\hvrpSol\sum_{l = 0}^{\abs{R}} \text{freq}(r_{hl}, r_{hl + 1}, k) \mathbb{I}[r_{hl}, r_{hl + 1} \notin \text{current solution}]$.
Here $\text{w}\hvrpSol$ refers to the weight of solution $\hvrpSol$ while $\text{freq}(i, j, k)$ refers to the frequency with which edge $(i, j) \in E$ appears among all routes of vehicle type $k$. Specifically, these quantities are defined as follows:
\begingroup
\allowdisplaybreaks
\begin{align*}
&\text{w}\hvrpSol = \dfrac{ \max_{\hvrpSolx{\prime} \in \mathcal{P}} \bar{c}\hvrpSolx{\prime} - \bar{c}\hvrpSol}{\max_{\hvrpSolx{\prime} \in \mathcal{P}} \bar{c}\hvrpSolx{\prime} - \min_{\hvrpSolx{\prime} \in \mathcal{P}} \bar{c}\hvrpSolx{\prime}},
\\
&\text{freq}(i, j, k) = \sum_{\hvrpSolx{\prime} \in \mathcal{P}} \sum_{h = 1}^{H'} \mathbb{I}[\kappa^\prime_h = k] \sum_{l = 0}^{\abs{R^\prime_h}} \mathbb{I}[(r^\prime_{hl}, r^\prime_{hl+1}) = (i, j)].
\end{align*}
\endgroup
The route $R$ with the highest score is then determined to be the candidate route of vehicle type $k$ (after deleting those customers that have already been routed in the current solution).

Among all generated routes, the candidate route of the vehicle type with the lowest value of \textit{average cost per unit of carried load} is then adopted in the current solution, and the entire procedure repeats.
At the end, if unrouted customers still remain, then they are inserted into an existing route (and corresponding position) for which a randomly weighted sum of the capacity violation and insertion cost is minimized, similar to the basic construction heuristic. 
We recall that this is done so that the fleet availability constraints~\ref{def:feasible:fleet_size} are never violated.

\paragraph{Reference set update method.}
In the initialization phase, the reference set $\mathcal{P}$ is grown to contain up to $\mu$ different solutions. In the exploitation phase, the size of $\mathcal{P}$ is kept constant by replacing older solutions with more recently encountered ones. To ensure an appropriate balance between quality and diversity among the reference solutions, the procedure \textsc{Update Reference Set}$(\mathcal{P}, \hvrpSol)$ uses a simple rule that replaces the worst solution (in terms of its total cost) $\hvrpSolx{W}$ whenever the candidate solution to be inserted $\hvrpSol$ satisfies $\bar{c}\hvrpSol < \bar{c}\hvrpSolx{W}$.

\section{Detailed Tables of Results}
The following tables report the best solutions found for each of the benchmark instances and uncertainty sets that we have considered in the paper.
The budget sets, factor models and general ellipsoids correspond to the setting of $(\alpha, \beta) = (0.1, 0.5)$, the cardinality-constrained and discrete sets correspond to $(\alpha, \beta) = (0.1, 0.2)$, while the axis-parallel ellipsoid corresponds to $(\alpha, \beta) = (0.1, 1.0)$.
In each table, ``Inst'' denotes the instance numbered as per the original dataset, $n$ and $m$ denote the number of customers and vehicle types respectively, while the quantity reported under uncertainty set $\mathcal{Q}$ is the best found solution for that instance and setting of the uncertainty set.

% Table generated by Excel2LaTeX from sheet 'detailed'
\begin{table}[htbp]
  \centering
  \small
  \caption{Summary of results for the HVRPFD instances.}
    \begin{tabularx}{\textwidth}{lrc*{7}{R}}
    \toprule
    Inst  & $n$    & $m$    & $\left\{q^0\right\}$ & $\mathcal{Q}_B$ & $\mathcal{Q}_F$ & $\mathcal{Q}_E^\text{ax}$ & $\mathcal{Q}_E^\text{gen}$ & $\mathcal{Q}_G$ & $\mathcal{Q}_S$ \\
    \midrule
    13     & 50     & 6      &    2,929.54  &      3,183.32  &      3,136.73  &    3,061.73  &    3,093.55  &      3,185.09  &    3,047.35  \\
    14     & 50     & 3      &    9,584.67  &    10,106.67  &    10,103.02  &    9,600.38  &    9,605.91  &    10,106.67  &    9,599.48  \\
    15     & 50     & 3      &    2,761.41  &      3,065.29  &      2,965.21  &    2,941.70  &    2,941.70  &      3,065.29  &    2,934.85  \\
    16     & 50     & 3      &    3,085.06  &      3,265.41  &      3,238.55  &    3,145.47  &    3,221.38  &      3,265.41  &    3,134.01  \\
    17     & 75     & 4      &    1,960.59  &      2,076.96  &      2,067.28  &    2,001.71  &    2,013.99  &      2,076.96  &    1,991.14  \\
    18     & 75     & 6      &    3,524.34  &      3,748.68  &      3,709.86  &    3,628.33  &    3,662.54  &      3,745.10  &    3,618.75  \\
    19     & 100    & 3      &    9,693.23  &    10,420.34  &    10,420.34  &    9,701.73  &    9,748.98  &    10,420.34  &    9,701.64  \\
    20     & 100    & 3      &    4,469.86  &      4,795.14  &      4,731.65  &    4,599.16  &    4,714.54  &      4,834.17  &    4,612.88  \\
    \bottomrule
    \end{tabularx}%
  \label{table:detailed_hvrpfd}%
\end{table}%

% Table generated by Excel2LaTeX from sheet 'detailed'
\begin{table}[htbp]
  \centering
  \small
  \caption{Summary of results for the HVRPD instances.}
    \begin{tabularx}{\textwidth}{lrc*{7}{R}}
    \toprule
    Inst  & $n$    & $m$    & $\left\{q^0\right\}$ & $\mathcal{Q}_B$ & $\mathcal{Q}_F$ & $\mathcal{Q}_E^\text{ax}$ & $\mathcal{Q}_E^\text{gen}$ & $\mathcal{Q}_G$ & $\mathcal{Q}_S$ \\
    \midrule
    13     & 50     & 6      &    1,432.57  &      1,517.84  &      1,502.24  &    1,481.13  &    1,482.49  &      1,517.84  &    1,468.83  \\
    14     & 50     & 3      &       584.38  &          606.67  &          603.02  &       596.53  &       597.54  &          606.67  &       596.63  \\
    15     & 50     & 3      &       961.41  &      1,015.29  &      1,012.64  &       991.70  &       991.70  &      1,015.29  &       984.85  \\
    16     & 50     & 3      &    1,085.06  &      1,144.94  &      1,133.16  &    1,120.57  &    1,121.38  &      1,144.94  &    1,110.41  \\
    17     & 75     & 4      &    1,009.48  &      1,061.96  &      1,052.28  &    1,021.50  &    1,023.99  &      1,061.96  &    1,020.28  \\
    18     & 75     & 6      &    1,761.88  &      1,823.58  &      1,812.32  &    1,783.93  &    1,788.38  &      1,823.58  &    1,777.02  \\
    19     & 100    & 3      &    1,093.23  &      1,120.34  &      1,120.34  &    1,101.64  &    1,116.47  &      1,120.34  &    1,101.64  \\
    20     & 100    & 3      &    1,472.97  &      1,534.17  &      1,533.62  &    1,501.53  &    1,514.39  &      1,534.17  &    1,508.31  \\
    \bottomrule
    \end{tabularx}%
  \label{table:detailed_hvrpd}%
\end{table}%

% Table generated by Excel2LaTeX from sheet 'detailed'
\begin{table}[htbp]
  \centering
  \small
  \caption{Summary of results for the FSMFD instances.}
    \begin{tabularx}{\textwidth}{lrc*{7}{R}}
    \toprule
    Inst  & $n$    & $m$    & $\left\{q^0\right\}$ & $\mathcal{Q}_B$ & $\mathcal{Q}_F$ & $\mathcal{Q}_E^\text{ax}$ & $\mathcal{Q}_E^\text{gen}$ & $\mathcal{Q}_G$ & $\mathcal{Q}_S$ \\
    \midrule
    3      & 20     & 5      &       986.93  &      1,144.22  &      1,118.76  &    1,106.84  &    1,106.65  &      1,144.22  &    1,092.59  \\
    4      & 20     & 3      &    6,377.28  &      6,437.33  &      6,432.25  &    6,413.06  &    6,413.06  &      6,437.33  &    6,392.34  \\
    5      & 20     & 5      &    1,167.40  &      1,322.26  &      1,298.72  &    1,287.48  &    1,287.48  &      1,322.26  &    1,249.08  \\
    6      & 20     & 3      &    6,416.11  &      6,516.47  &      6,500.82  &    6,499.74  &    6,499.74  &      6,516.47  &    6,470.82  \\
    13     & 50     & 6      &    2,711.61  &      2,964.65  &      2,935.40  &    2,906.68  &    2,873.24  &      2,964.65  &    2,908.96  \\
    14     & 50     & 3      &    8,582.05  &      9,126.90  &      8,644.00  &    8,618.16  &    8,618.16  &      9,126.90  &    8,618.16  \\
    15     & 50     & 3      &    2,450.82  &      2,634.96  &      2,608.61  &    2,590.47  &    2,591.86  &      2,634.96  &    2,591.93  \\
    16     & 50     & 3      &    2,906.71  &      3,168.92  &      3,137.50  &    3,061.09  &    3,070.34  &      3,168.92  &    3,052.39  \\
    17     & 75     & 4      &    1,872.49  &      2,004.48  &      1,969.64  &    1,932.82  &    1,946.09  &      2,004.48  &    1,925.47  \\
    18     & 75     & 6      &    2,918.45  &      3,153.09  &      3,100.08  &    3,056.88  &    3,064.50  &      3,152.16  &    3,063.70  \\
    19     & 100    & 3      &    8,094.97  &      8,662.86  &      8,649.99  &    8,132.14  &    8,409.14  &      8,662.86  &    8,134.19  \\
    20     & 100    & 3      &    3,839.11  &      4,165.91  &      4,138.25  &    4,046.05  &    4,085.22  &      4,168.44  &    4,054.18  \\
    \bottomrule
    \end{tabularx}%
  \label{table:detailed_fsmfd}%
\end{table}%

% Table generated by Excel2LaTeX from sheet 'detailed'
\begin{table}[htbp]
  \centering
  \small
  \caption{Summary of results for the FSMF instances.}
    \begin{tabularx}{\textwidth}{lrc*{7}{R}}
    \toprule
    Inst  & $n$    & $m$    & $\left\{q^0\right\}$ & $\mathcal{Q}_B$ & $\mathcal{Q}_F$ & $\mathcal{Q}_E^\text{ax}$ & $\mathcal{Q}_E^\text{gen}$ & $\mathcal{Q}_G$ & $\mathcal{Q}_S$ \\
    \midrule
    3      & 20     & 5      &       887.99  &          954.37  &          949.79  &       921.23  &       927.96  &          951.61  &       918.63  \\
    4      & 20     & 3      &    6,327.60  &      6,437.33  &      6,432.25  &    6,413.06  &    6,413.06  &      6,437.33  &    6,379.27  \\
    5      & 20     & 5      &       919.86  &      1,005.27  &          980.57  &       961.63  &       963.63  &          988.63  &       949.85  \\
    6      & 20     & 3      &    6,353.72  &      6,516.47  &      6,500.82  &    6,499.74  &    6,499.74  &      6,516.47  &    6,448.52  \\
    13     & 50     & 6      &    2,223.70  &      2,399.20  &      2,365.21  &    2,305.17  &    2,327.57  &      2,406.36  &    2,304.11  \\
    14     & 50     & 3      &    8,562.41  &      9,119.03  &      8,644.00  &    8,618.16  &    8,618.16  &      9,119.03  &    8,618.16  \\
    15     & 50     & 3      &    2,360.49  &      2,586.37  &      2,561.65  &    2,475.63  &    2,492.99  &      2,586.37  &    2,463.19  \\
    16     & 50     & 3      &    2,506.66  &      2,721.40  &      2,697.88  &    2,610.03  &    2,646.51  &      2,720.43  &    2,612.59  \\
    17     & 75     & 4      &    1,621.25  &      1,734.53  &      1,712.24  &    1,658.29  &    1,680.08  &      1,734.53  &    1,657.43  \\
    18     & 75     & 6      &    2,195.13  &      2,369.65  &      2,322.62  &    2,272.55  &    2,294.51  &      2,369.65  &    2,267.88  \\
    19     & 100    & 3      &    8,094.97  &      8,662.86  &      8,649.99  &    8,129.33  &    8,382.79  &      8,662.86  &    8,135.02  \\
    20     & 100    & 3      &    3,714.44  &      4,043.97  &      3,973.46  &    3,862.99  &    3,911.89  &      4,060.73  &    3,874.70  \\
    \bottomrule
    \end{tabularx}%
  \label{table:detailed_fsmf}%
\end{table}%

% Table generated by Excel2LaTeX from sheet 'detailed'
\begin{table}[htbp]
  \centering
  \small
  \caption{Summary of results for the FSMD instances.}
    \begin{tabularx}{\textwidth}{lrc*{7}{R}}
    \toprule
    Inst  & $n$    & $m$    & $\left\{q^0\right\}$ & $\mathcal{Q}_B$ & $\mathcal{Q}_F$ & $\mathcal{Q}_E^\text{ax}$ & $\mathcal{Q}_E^\text{gen}$ & $\mathcal{Q}_G$ & $\mathcal{Q}_S$ \\
    \midrule
    3      & 20     & 5      &       555.66  &          623.22  &          609.51  &       604.02  &       604.02  &          623.22  &       585.02  \\
    4      & 20     & 3      &       369.77  &          378.70  &          373.24  &       369.77  &       373.24  &          380.71  &       373.24  \\
    5      & 20     & 5      &       713.62  &          742.87  &          737.36  &       736.90  &       736.90  &          742.85  &       721.00  \\
    6      & 20     & 3      &       391.42  &          414.66  &          415.03  &       396.26  &       405.46  &          406.19  &       396.26  \\
    13     & 50     & 6      &    1,405.87  &      1,491.86  &      1,488.78  &    1,468.28  &    1,471.82  &      1,491.86  &    1,460.06  \\
    14     & 50     & 3      &       575.63  &          603.21  &          601.71  &       583.94  &       588.64  &          603.21  &       583.94  \\
    15     & 50     & 3      &       944.96  &          999.82  &          997.98  &       979.40  &       985.19  &          999.82  &       974.59  \\
    16     & 50     & 3      &    1,078.67  &      1,131.00  &      1,114.78  &    1,110.88  &    1,111.69  &      1,131.00  &    1,107.74  \\
    17     & 75     & 4      &    1,006.50  &      1,036.54  &      1,029.75  &    1,017.52  &    1,020.50  &      1,038.60  &    1,016.64  \\
    18     & 75     & 6      &    1,746.53  &      1,800.80  &      1,795.02  &    1,778.61  &    1,785.58  &      1,800.80  &    1,775.35  \\
    19     & 100    & 3      &    1,073.71  &      1,106.17  &      1,101.07  &    1,084.54  &    1,099.69  &      1,105.44  &    1,093.74  \\
    20     & 100    & 3      &    1,468.79  &      1,530.52  &      1,525.33  &    1,495.31  &    1,514.46  &      1,533.24  &    1,497.01  \\
    \bottomrule
    \end{tabularx}%
  \label{table:detailed_fsmd}%
\end{table}%

% Table generated by Excel2LaTeX from sheet 'detailed'
\begin{table}[htbp]
  \centering
  \small
  \caption{Summary of results for the SDVRP instances.}
    \begin{tabularx}{\textwidth}{lrc*{7}{R}}
    \toprule
    Inst  & $n$    & $m$    & $\left\{q^0\right\}$ & $\mathcal{Q}_B$ & $\mathcal{Q}_F$ & $\mathcal{Q}_E^\text{ax}$ & $\mathcal{Q}_E^\text{gen}$ & $\mathcal{Q}_G$ & $\mathcal{Q}_S$ \\
    \midrule
    1      & 50     & 3      &       609.52  &       640.32  &       634.19  &       621.50  &       624.77  &       640.32  &       623.37  \\
    2      & 50     & 2      &       563.85  &       598.10  &       593.59  &       569.45  &       577.70  &       598.10  &       569.45  \\
    3      & 75     & 3      &       899.90  &       954.32  &       938.11  &       908.95  &       922.33  &       954.32  &       908.95  \\
    4      & 75     & 2      &       806.72  &       854.43  &       843.73  &       831.60  &       836.25  &       854.43  &       825.27  \\
    5      & 100    & 3      &       963.09  &    1,003.57  &       993.64  &       977.14  &       978.00  &    1,003.57  &       976.01  \\
    6      & 100    & 2      &       965.86  &    1,028.52  &    1,016.95  &       987.88  &    1,000.59  &    1,028.52  &       983.46  \\
    7      & 27     & 3      &       391.30  &       391.30  &       391.30  &       391.30  &       391.30  &       391.30  &       391.30  \\
    8      & 54     & 3      &       664.46  &       664.46  &       664.46  &       664.46  &       664.46  &       664.46  &       664.46  \\
    9      & 81     & 3      &       948.23  &       948.23  &       948.23  &       948.23  &       948.23  &       948.23  &       948.23  \\
    10     & 108    & 3      &    1,189.69  &    1,218.75  &    1,218.75  &    1,208.74  &    1,208.74  &    1,218.75  &    1,200.34  \\
    13     & 54     & 3      &    1,150.31  &    1,194.18  &    1,194.18  &    1,162.85  &    1,171.58  &    1,194.18  &    1,171.58  \\
    14     & 108    & 3      &    1,873.74  &    1,960.88  &    1,960.88  &    1,886.53  &    1,908.76  &    1,960.88  &    1,889.59  \\
    23     & 100    & 3      &       764.19  &       803.29  &       803.29  &       781.81  &       784.83  &       803.29  &       781.81  \\
    \bottomrule
    \end{tabularx}%
  \label{table:detailed_sdvrp}%
\end{table}%

% Table generated by Excel2LaTeX from sheet 'detailed'
\begin{table}[htbp]
  \centering
  \small
  \caption{Summary of results for the MDVRP instances.}
    \begin{tabularx}{\textwidth}{lrc*{7}{R}}
    \toprule
    Inst  & $n$    & $m$    & $\left\{q^0\right\}$ & $\mathcal{Q}_B$ & $\mathcal{Q}_F$ & $\mathcal{Q}_E^\text{ax}$ & $\mathcal{Q}_E^\text{gen}$ & $\mathcal{Q}_G$ & $\mathcal{Q}_S$ \\
    \midrule
    1      & 50     & 4      &       560.63  &       576.87  &       570.81  &       570.81  &       570.81  &       576.87  &       575.36  \\
    2      & 50     & 4      &       464.35  &       473.53  &       473.01  &       464.48  &       464.48  &       473.22  &       464.48  \\
    3      & 75     & 5      &       623.79  &       640.65  &       635.93  &       629.86  &       630.98  &       641.19  &       629.67  \\
    4      & 100    & 2      &       951.28  &       999.21  &       997.20  &       973.68  &       982.38  &       999.21  &       978.74  \\
    5      & 100    & 2      &       725.55  &       750.03  &       747.29  &       731.27  &       741.44  &       750.03  &       732.17  \\
    6      & 100    & 3      &       838.00  &       876.50  &       871.47  &       864.49  &       867.39  &       877.26  &       865.93  \\
    7      & 100    & 4      &       843.30  &       881.97  &       871.31  &       855.39  &       860.65  &       881.97  &       855.55  \\
    12     & 80     & 2      &    1,290.93  &    1,318.95  &    1,314.36  &    1,314.36  &    1,314.36  &    1,318.95  &    1,301.68  \\
    15     & 160    & 4      &    2,463.14  &    2,505.42  &    2,505.42  &    2,505.42  &    2,505.42  &    2,505.42  &    2,492.55  \\
    \bottomrule
    \end{tabularx}%
  \label{table:detailed_mdvrp}%
\end{table}%
\end{appendices}

\end{document}